\documentclass{amsart}
\usepackage{amsmath,amssymb,latexsym, amsfonts, amscd, amsthm}

\usepackage{hyperref}
\hypersetup{colorlinks=false}
\usepackage{graphicx}

\usepackage{bbm}
\usepackage{enumerate}
\usepackage[all]{xy}

\newtheorem{thm}{Theorem}[section]
\newtheorem{pro}[thm]{Proposition}
\newtheorem{lm}[thm]{Lemma}

\numberwithin{equation}{section}
\newtheorem{wh}[thm]{Working Hypothesis}

\theoremstyle{definition}

\newtheorem{exa}[thm]{Example}

\theoremstyle{remark}
\newtheorem{rem}[thm]{Remark}

\numberwithin{equation}{section}

\DeclareMathOperator*{\Irr}{Irr}
\DeclareMathOperator*{\disc}{disc}
\DeclareMathOperator*{\Gal}{Gal}
\DeclareMathOperator*{\Res}{Res}

\DeclareMathOperator*{\Ind}{Ind}

\DeclareMathOperator*{\id}{id}

\DeclareMathOperator*{\temp}{temp}
\DeclareMathOperator*{\ad}{ad}
\DeclareMathOperator*{\scn}{sc}

\DeclareMathOperator*{\der}{der}

\DeclareMathOperator*{\Hom}{Hom}
\DeclareMathOperator*{\Aut}{Aut}

\DeclareMathOperator*{\imag}{Im}

\newcommand{\mcA}{\mathcal{A}}
\newcommand{\vp}{\varphi}
\newcommand{\tvp}{\widetilde{\varphi}}
\newcommand{\s}{\simeq}

\newcommand{\mm}{\mathfrak{m}}
\newcommand{\ee}{\mathfrak{e}}

\newcommand{\si}{\sigma}
\newcommand{\ts}{\widetilde{\sigma}}

\newcommand{\CC}{\mathbb{C}}
\newcommand{\NN}{\mathbb{N}}

\newcommand{\QQ}{\mathbb{Q}}
\newcommand{\ZZ}{\mathbb{Z}}

\def\bA{\bold A}
\def\bG{\bold G}

\def\bP{\bold P}

\def\bB{\bold B}
\def\bT{\bold T}

\def\bA{\bold A}
\def\bG{\bold G}
\def\bP{\bold P}

\def\bZ{\bold Z}

\def\bU{\bold U}

\def\cS{\mathcal S}
\def\L{\mathcal L}

\DeclareMathOperator*{\Sp}{Sp}

\DeclareMathOperator*{\SL}{SL}
\DeclareMathOperator*{\GL}{GL}

\DeclareMathOperator*{\GSp}{GSp}

\DeclareMathOperator*{\PGL}{PGL}
\DeclareMathOperator*{\GSpin}{GSpin}

\newcommand{\tG}{\widetilde{G}}
\newcommand{\tbG}{\widetilde{\bold G}}

\DeclareMathOperator*{\sgn}{\textsf{sgn}}

\def\cS{\mathcal S}
\def\L{\mathcal L}

\newcommand{\tdel}{\widetilde{\delta}}
\newcommand{\tsigma}{\widetilde{\sigma}}

\newcommand{\trho}{\widetilde{\rho}}

\newcommand{\Z}{\mathbf{Z}}

\newcommand{\G}{\mathbf{G}}

\newcommand{\tH}{\widetilde{H}}

\begin{document}

\title[On multiplicity in restriction of tempered representations]{On multiplicity in restriction of tempered representations of $p$-adic groups}

\author[Kwangho Choiy]{Kwangho Choiy} 
\address{Department of Mathematics,
Southern Illinois University,
Carbondale, IL 62901-4408,
U.S.A.}
\email{kchoiy@siu.edu}



\date{\today 
}

\begin{abstract}
We establish an equality between two multiplicities: one in the restriction of tempered representations of a $p$-adic group to its closed subgroup with the same derived group; and one occurring in their corresponding component groups in Langlands dual sides, so-called $\mathcal{S}$-groups, under working hypotheses about the tempered local Langlands conjecture and the internal structure of tempered $L$-packets.
This provides a formula of the multiplicity for $p$-adic groups by means of dimensions of irreducible representations of their $\mathcal{S}$-groups.
\end{abstract}

\maketitle

\setcounter{tocdepth}{1}
\tableofcontents
\section{Introduction} \label{intro}
From the construction of $L$-packets of the non-split inner form of $\SL_2$ over a $p$-adic field $F$ of characteristic 0 by Labesse and Langlands \cite{ll79, shel79}, it is observed that the multiplicity in the restriction of irreducible smooth representations of the non-split inner form of $\GL_2$ to the non-split inner form of $\SL_2$ fails to be one. 
Due to the uniqueness of Whittaker models \cite{rod}, on the other hand, the multiplicity is always one for the split form $\SL_n$ with any positive integer $n$ \cite{gk82, tad92}.
Extending the result of Labesse and Langlands to any inner form of $\SL_n$ over $F,$
Hiraga and Saito discovered in \cite{hs11} that
the multiplicity in restriction from inner forms of $\GL_n$ to those of $\SL_n$ equals the dimension of an irreducible representation of a component group (so-called, $\cS$-group) in $\SL_n(\CC)$ via the local Langlands correspondence for $F$-inner forms of $\SL_n.$ Adler and Prasad established some multiplicity-one theorems in the restriction from $GU(V)$ to $U(V)$ for the cases with no Whittaker model by means of Fourier-Jacobi models, where $V$ is a finite-dimensional vector space over a local field of characteristic not $2$ with a non-degenerate symmetric or skew-symmetric form \cite{ap06}.
B. Xu also studied the multiplicity for quasi-split groups under some assumptions \cite{xu15}.
Furthermore, the author's work on the local Langlands correspondence for the $F$-inner form $\Sp_{1,1}$ of $\Sp_4$ offers an interesting example \cite[Section 7.7]{ch15}, where the multiplicity in restriction of an irreducible smooth representation from $\GSp_{1,1}$ to $\Sp_{1,1}$ no longer coincides with the dimension of the corresponding irreducible representation of the $\cS$-group in $\Sp_4(\CC).$ The author and his collaborators also addressed the multiplicity for irreducible smooth representations of small rank general spin groups $\GSpin_4,$ $\GSpin_6$ and their inner forms in \cite{acgspin}, and that for unitary principal series representations of $\GSpin_n$ in \cite{bcgmurthy, bcg_tentative}.

Inspired by such phenomena, we study the multiplicity in the restriction of irreducible smooth representations of $\widetilde{\bold G}(F)$ to $\bold G(F)$ in a general setting, where $\widetilde{\bold G}$ is a connected reductive algebraic $F$-group and $\bold G$ is its closed $F$-subgroup with the same derived group.
The purpose of this paper is to establish an equality of the multiplicity in the restriction of tempered representations from $\tbG(F)$ to $\bG(F)$ and the multiplicity in the restriction of corresponding representations from the associated $\cS$-group of $\bG$ to that of $\tbG,$ under the hypotheses, listed in Section \ref{main}, about the tempered local Langlands conjectures for $\tbG$ and $\bG$ and the conjectural structures of their $L$-packets.
This equality yields a formula of the multiplicity in the restriction that generalizes the previous results mentioned above.

To be precise, we let $\tbG$ be a connected reductive group over $F,$ 
and let $\bG$ be a closed $F$-subgroup of $\tbG$ such that
\begin{equation}  \label{cond on G introduction}
\bG_{\der} = \tbG_{\der} \subseteq \bG \subseteq \tbG,
\end{equation}  
where the subscript ${\der}$ stands for the derived group.
We write $G=\bG(F)$ and $\tG=\tbG(F)$ for the groups of $F$-points. 
Given irreducible smooth representations $\sigma \in \Irr(G)$ and $\pi \in \Irr(\tG),$ 
\textit{the multiplicity} $\langle \sigma, \pi \rangle_{G}$ of $\sigma$ in the restriction ${\Res}_{G}^{\tG} (\pi)$ of $\pi$ to $G$ is defined as follows:
\[
\langle \sigma, \pi \rangle_{G}:= \dim_{\CC}
\; {\Hom}_{G}(\sigma, {\Res}_{G}^{\tG} (\pi)).
\]
For any finite groups $H \leq \tH,$ given $\delta \in \Irr(H)$ and $\rho \in \Irr(\tH),$ we shall define \textit{the multiplicity} $\langle \delta, \rho \rangle_{H}$ of $\delta$ in the restriction ${\Res}_{H}^{\tH} (\rho)$ of $\rho$ to $H$ as follows:
\[
\langle \delta, \rho \rangle_{H}:= \dim_{\CC}
\; {\Hom}_{H}(\delta, {\Res}_{H}^{\tH} (\rho)).
\]

Let $\vp$ be given in the set $\Phi_{\temp}(G)$ of $\widehat{\bG}$-conjugacy classes of tempered $L$-parameters. 
Here $\widehat{\bG} = {^L}\bG^0$ denotes  the complex dual of $\bG$ \cite{bo79}. 
Choose a lifting $\tvp \in \Phi_{\temp}(\tG)$ of $\vp$ using Labesse's theorem  (See Theorem \ref{thm by Labesse}).
We here assume that the local Langlands conjecture and the conjectural structure of $L$-packets are valid for tempered representations of $\tG$ and $G$ in \eqref{cond on G introduction}. 
We refer the reader to Working Hypotheses \ref{wh temp llc}, \ref{shahidi's conj}, and \ref{last hyp} for precise statements regarding these conjectures.
Extending a lemma of Chao and Li \cite{chaoli}, we have
\begin{equation}  \label{crucial exact seq intro}
1 \longrightarrow \cS_{\tvp, \scn}(\widehat{\tbG})  \longrightarrow \cS_{\vp, \scn}(\widehat{\bG})  \longrightarrow X(\tvp) \longrightarrow 1,
\end{equation}
where $X(\tvp)$ is the set of certain cohomological classes stabilizing $ \tvp.$ 
Given a tempered representation $\sigma \in \Pi_{\vp}(G),$ we fix a tempered representation $\ts \in \Pi_{\tvp}(\tG)$  such that $\sigma \subset \Res_G^{\tG} \ts.$ 
The internal structures of $L$-packets for $G$ and $\tG$ provide two finite representations $\rho \in \Irr(\cS_{\vp, \scn}(\widehat \bG), \zeta_{\bG} )$ and $\trho \in  \Irr(\cS_{\tvp, \scn}(\widehat{\tbG}), \zeta_{\tbG} )$ corresponding to $\si$ and $\ts,$ respectively. 
Here 
$\Irr(\cS_{\vp, \scn}, \zeta_{G})$ denotes the set of irreducible representations of $\cS_{\vp, \scn}$ with a certain requirement with respect to a character $\zeta_{G}.$ The relevant definitions will be given in Section \ref{main}.
Note that the condition 
\eqref{cond on G introduction} yields $\zeta_{\bG} =\zeta_{\tbG} .$
For simplicity of notation, we set $\cS_{\vp, \scn}(\widehat \bG)=\cS_{\vp, \scn}$ and $\cS_{\tvp, \scn}(\widehat \tbG)=\cS_{\tvp, \scn}.$
Based on \eqref{crucial exact seq intro}, we construct a homomorphism 
\[  
\Lambda: \tG/G \twoheadrightarrow (\cS_{\vp, \scn}/\cS_{\tvp, \scn})^\vee,
\]
where $(\cS_{\vp, \scn}/\cS_{\tvp, \scn})^\vee$ denotes the group of characters on the finite quotient group $\cS_{\vp, \scn}/\cS_{\tvp, \scn}.$ 
This homomorphism yields an isomorphism
\begin{equation} \label{1st bijection intro}
\lambda: \tG/\tG_{\si} \overset{\s}{\longrightarrow} (\cS_{\vp, \scn}/\cS_{\tvp, \scn})^\vee \big/ I(\rho),
\end{equation} 
where $\tG_{\si} =: \{g \in \tG : {^g}\si \s \si  \},$ and
 $I(\rho) := \{\chi \in (\cS_{\vp, \scn}/\cS_{\tvp, \scn})^\vee : \rho \chi \s \rho \}$ (Theorem \ref{thm 1}). 
For this bijection, we impose one further condition (Working Hypothesis \ref{llast hyp})
that the conjectural internal structure of $L$-packet of $G$ allows ${^g}\sigma$ to correspond to $\rho \chi_g,$ along with Working Hypotheses \ref{wh temp llc}, \ref{shahidi's conj}, and \ref{last hyp}. Here, the character $\chi_g \in (\cS_{\vp, \scn}/\cS_{\tvp, \scn})^\vee$ is given by $\Lambda.$
Furthermore, as an analogue of $\lambda,$ we also have the following isomorphism
\begin{equation} \label{2nd bijection intro}
\widehat \lambda: \cS_{\vp, \scn}/(\cS_{\vp, \scn})_{\trho}  \overset{\s}{\longrightarrow} X(\tvp)/\bar I(\ts), 
\end{equation}
where $(\cS_{\vp, \scn})_{\trho} = \{s \in \cS_{\vp, \scn} : {^s}\trho \s \trho \},$ and 
$\bar I(\ts)$ denotes the image of $I(\ts):= \{ \chi \in (\tG/G)^\vee : \ts \s \ts \chi \}$ in $X(\tvp)$ via the maps $I(\ts) \hookrightarrow \{ a \in H^1(W_F, \widehat {(\tbG/\bG)}) : a \tvp \s \tvp \mbox{ in } \widehat{\tbG} \} \twoheadrightarrow X(\tvp)$
(Theorem \ref{thm 2}). Likewise, for this bijection, we need to assume Working Hypothesis \ref{really llast hyp}
that  the conjectural internal structure of $L$-packet of $\tG$ allows
$\tsigma \chi_s$ to correspond to ${^s}\trho,$ along with Working Hypotheses \ref{wh temp llc}, \ref{shahidi's conj}, and \ref{last hyp}. 
Here, the character  $\chi_s \in (\tG/G)^\vee$  is given by the composition of
 the isomorphism $X(\tvp) \s  \cS_{\vp, \scn}/\cS_{\tvp, \scn}$ (see \eqref{imp exact}) and the map $X(\tvp) \rightarrow  (\tG/G)^\vee$ (see \eqref{for pre hom}).

Recalling that the $L$-packet $\Pi_{\vp}(G) $ is parameterized by $\Irr(\cS_{\vp, \scn}, \zeta_{G}),$ we note that its subset $\Pi_{\ts}(G)$ consisting of all irreducible inequivalent constituents in the restriction ${\Res}_{G}^{\tG}(\ts)$ of $\ts$ from $\tG$ to $G$ is controlled by the quotient $(\cS_{\vp, \scn}/\cS_{\tvp, \scn})^\vee.$
We further remark that the quotient $\cS_{\vp, \scn}/(\cS_{\vp, \scn})_{\trho}$  in \eqref{2nd bijection intro} (hence, the set $\Pi_{\rho}(\cS_{\tvp, \scn})$) is in bijection with 
the $X(\tvp)$-orbit, $\{\ts \chi_a : a \in X(\tvp) \},$ of $\ts$ in  $\Pi_{\tvp}(\tG)$ (see Remark \ref{rem for X-orbit}).
It should be mentioned that $\lambda$ and $\widehat \lambda$ both rely on the conjectural bijection \eqref{internal st} characterizing the internal structures of $L$-packets $\Pi_{\tvp}(\tG)$ and $\Pi_{\vp}(\G)$ (see Remarks \ref{rem for int structure} and \ref{remark for lambda and lambda hat}).

Based on two bijections \eqref{1st bijection intro} and \eqref{2nd bijection intro}, we finally prove the following equality
\begin{equation} \label{multi equal intro}
\langle \sigma, \ts \rangle_{G} = \langle \trho, \rho \rangle_{\cS_{\tvp, \scn}}
\end{equation}
(Theorem \ref{mult thm}).  
Using Clifford theory and some results in the restriction of representations of finite groups, the equality \eqref{multi equal intro} yields 
\begin{equation} \label{multi formula intro}
\langle \sigma, \ts \rangle_{G} = \frac{\dim \rho}{\dim \widetilde\delta} |\Pi_{\rho}(\cS_{\tvp, \scn})|^{-1}
\end{equation}
for any $\widetilde\delta \in \Irr(\cS_{\tvp, \scn})$ satisfying $\widetilde\delta \subset {\Res}_{\cS_{\tvp, \scn}}^{\cS_{\vp, \scn}}(\rho)$ (Theorem \ref{multi generalized}). 
Hence, the multiplicity in the restriction in $p$-adic groups is formulated by means of dimensions of irreducible representations of their $\cS$-groups. 
These $\cS$-groups are all finite but not necessarily abelian.
We remark that, since $\trho \subset {\Res}_{\cS_{\tvp, \scn}}^{\cS_{\vp, \scn}}(\rho),$ the formula \eqref{multi formula intro} is the generalization of Hiraga and Saito's related work in \cite{hs11} for the case of $\tG=\GL_m(D)$ and $G=\SL_m(D)$ (see \eqref{dim=dim=multi}), where $\dim \trho = |\Pi_{\rho}(\cS_{\tvp, \scn})|=1$ and $\langle \sigma, \ts \rangle_{G} =\dim \rho$ (see Section \ref{section for SL}). Furthermore, the formula \eqref{multi formula intro} extends relevant results in \cite{keys87, xu15} from some quasi-split settings to the non-quasi-split (see Remark \ref{rem related to xu}).

In Section \ref{backgrounds}, we recall basic notions and terminologies. Also, the local Langlands conjecture for a $p$-adic group and the conjectural structure of $L$-packets are reviewed and some useful arguments are discussed.
In Section \ref{restriction on p-adic and finite}, we review some results of Gelbart-Knapp \cite{gk82}, Tadi\'c  \cite{tad92}, and Hiraga-Saito \cite{hs11}, regarding the restriction of representations of $p$-adic groups, and discuss some arguments related to the restriction of representations of finite groups, based on Clifford theory. Under working hypotheses about the local Langlands conjecture and the internal structure of $L$-packets, Section \ref{section of main results} provides two bijections \eqref{1st bijection intro} \& \eqref{2nd bijection intro}.
Furthermore, the equality \eqref{multi equal intro} of multiplicities in the two sides is established and the multiplicity formula \eqref{multi formula intro} is then formulated.
We also observe that Hiraga and Saito's work in \cite{hs11} about the multiplicity in the case of $\GL_m(D)$ and $\SL_m(D)$ with a central division algebra $D$ is generalized to arbitrary connected reductive groups $\tbG$ and $\bG$ in \eqref{cond on G introduction}.
In Appendix \ref{examples}, we provide some examples related to the results established in Section \ref{section of main results}.
\section{Basic notation and backgrounds} \label{backgrounds}
\subsection{Notation and conventions} \label{notation}
We take the following notation throughout the paper, unless otherwise specified. 
Let $p$ be a prime number. 
The field $F$ denotes a finite extension of $\QQ_p$ with an algebraic closure $\bar{F}.$ 
We denote by $W_F$ the Weil group of $F$ and by $\Gamma$ the absolute Galois group $\Gal(\bar{F} / F).$ 

Let $\bG$ be a connected reductive algebraic group over $F.$ 
We denote by $\bA_{\bG}$ the split component, that is, the maximal $F$-split torus in the center of $\bG.$ 
We denote by $G=\bG(F)$ the group of $F$-points. Fixing $\Gamma$-invariant splitting data, 
we define the $L$-group of $\bG$ as a semi-direct product $^{L}\bG := \widehat{\bG} \rtimes \Gamma$ (see \cite[Section 2]{bo79}). 
For $i \in \NN$ and $\Gamma$-module $A,$ we denote by $H^i(F, A) := H^i(\Gamma, A)$
the Galois cohomology of $A.$ 
Since any $\Gamma$-module $A$ can be given a $W_F$-module structure through the continuous homomorphism $W_F \rightarrow \Gamma$ with dense image, we have a restriction map $H^i(F,A) \rightarrow H^i(W_F,A)$ (see \cite{la85, karpuk13}).

Let $\Irr(G)$ denote the set of isomorphism classes of irreducible smooth complex representations of $G.$
By abuse of notation, we identify an isomorphism class with its representative. 
We often write $V_{\pi}$ for the space of $\pi \in \Irr(G).$  
We denote by $\Irr_{\disc}(G)$ the subset of $\Irr(G)$ consisting of discrete series representations, i.e., their central characters are unitary and the absolute values of their matrix coefficients are square-integrable modulo the center of $G,$ and by $\Irr_{\temp}(G)$ the subset of tempered representations of $G.$

For any topological group $J,$  
Write $Z(J)$ for its center, and $J ^\vee$ for the group $\Hom(J, \CC^{\times})$ of all continuous characters. 
We denote by $\mathbbm{1}$ the trivial character. 
The cardinality of a finite set $X$ is denoted by $|X|.$ For two integers $x$ and $y,$ $x \big{|} y$ means that $y$ is divisible by $x.$
\subsection{Definitions of multiplicities in restriction} \label{def section of multi}
Let $\tbG$ be a connected reductive group over $F,$ 
and let $\bG$ be a closed $F$-subgroup of $\tbG$ such that
\begin{equation}  \label{cond on G}
\bG_{\der} = \tbG_{\der} \subseteq \bG \subseteq \tbG,
\end{equation}  
where the subscript ${\der}$ stands for the derived group.
We write $G=\bG(F)$ and $\tG=\tbG(F)$ for the groups of $F$-points. 
Given irreducible smooth representations $\sigma \in \Irr(G)$ and $\pi \in \Irr(\tG),$ 
\textit{the multiplicity} $\langle \sigma, \pi \rangle_{G}$ of $\sigma$ in the restriction ${\Res}_{G}^{\tG} (\pi)$ of $\pi$ to $G$ is defined as follows:
\[
\langle \sigma, \pi \rangle_{G}:= \dim_{\CC}
\; {\Hom}_{G}(\sigma, {\Res}_{G}^{\tG} (\pi)).
\]

Let $\tH$ be a finite group and let $H$ be its subgroup. 
We denote by $\Irr(\sharp)$ the set of isomorphism classes of irreducible complex representations of a finite group $\sharp$ and identify an isomorphism class with its representative, by abuse of notation.
Given $\delta \in \Irr(H)$ and $\rho \in \Irr(\tH),$ we shall define \textit{the multiplicity} $\langle \delta, \rho \rangle_{H}$ of $\delta$ in the restriction ${\Res}_{H}^{\tH} (\rho)$ of $\rho$ to $H$ as follows:
\[
\langle \delta, \rho \rangle_{H}:= \dim_{\CC}
\; {\Hom}_{H}(\delta, {\Res}_{H}^{\tH} (\rho))
\]
(cf., Chapter 20 of \cite{jamesliebeck01}).
\subsection{$L$-parameters} \label{l-parameters}
Given a connected reductive algebraic group $\bG$ over $F,$ we let $\Phi(G)$ denote the set of $\widehat{\bG}$-conjugacy classes of $L$-parameters, i.e., admissible homomorphisms 
\[
\vp: W_F \times {\SL}_2(\CC) \longrightarrow {^L}\bG,
\]
(see \cite[Section 8.2]{bo79}). 
Here, we note from the definition of admissible homomorphisms that a condition of relevance is required for non quasi-split group $\bG,$ though any $L$-parameter for quasi-split groups is relevant to $\bG.$
Following \cite[Sections 3 and 8]{bo79}, a parabolic subgroup of ${^L}\bG$ is relevant to $\bG$ if it is equal to ${^L}\bP$ for some parabolic $F$-subgroup $\bP$ of $\bG.$ 
We say that $\vp \in \Phi(G)$ is relevant to $G$ if any parabolic subgroup of ${^L}\bG$ containing the image of $\vp$ is relevant to $\bG.$ 

Given $\vp \in \Phi(G),$ we denote by $S_{\vp}$ the centralizer in $\widehat{\bG}$ of the image of $\vp$ and $Z(\widehat{\bG})^{\Gamma}$ the $\Gamma$-invariant group of $Z(\widehat{\bG}).$
We say $\vp \in \Phi(G)$ is tempered if $\vp(W_F)$ is bounded, and $\vp$ is discrete if it is tempered and the quotient group $S_{\vp} / Z(\widehat{\bG})^{\Gamma}$ is finite.
Let $\Phi_{\disc}(G)$ and $\Phi_{\temp}(G)$ denote the subsets of $\Phi(G)$ consisting of discrete and tempered $L$-parameters of $G,$ respectively. 
We further say $\vp$ is elliptic if $S_{\vp} / Z(\widehat{\bG})^{\Gamma}$ is finite, equivalently if the image of $\vp$ in $\widehat G$ is contained in no proper parabolic subgroup.

\subsection{Labesse's theorem}
We discuss the following argument related to local Langlands conjecture in a certain special setting, which will be used in Sections \ref{section of main results} and \ref{section for SL}.
Let $\bG$ and $\tbG$ be connected reductive algebraic groups over $F$ satisfying the following exact sequence of connected components of $L$-groups
\begin{equation} \label{Labesse situation}
1 \longrightarrow \widehat{S} \longrightarrow \widehat{\tbG} \overset{pr}{\longrightarrow}\widehat{\bG} \longrightarrow 1,
\end{equation}
where $\widehat{S}$ is a central torus in $\widehat{\tbG},$ 
and $pr$ is a surjective homomorphism which is compatible with $\Gamma$-actions on $\widehat{\tbG}$ and $\widehat{\bG}.$
\begin{thm} (\cite[Th\'{e}or\`{e}m 8.1]{la85}) \label{thm by Labesse}
With the above notation, for any $\vp \in \Phi(G),$ there exists $\tvp \in \Phi(\tG)$ such that
\[
\vp = \tvp \circ pr.
\]
\end{thm} 
Such a parameter $\tvp$ is determined up to a $1$-cocycle of $W_F$ in $\widehat{S}$ (see \cite[Section 7]{la85} and \cite[Theorem 3.5.1]{chaoli}). 
Theorem \ref{thm by Labesse} tells us the existence of a lifting of a given $L$-parameter in the setting \eqref{Labesse situation}. 
\section{On restriction of representations} \label{restriction on p-adic and finite}
We recall some known results by Gelbart-Knapp in \cite{gk82}, Tadi\'c in \cite{tad92}, and Hiraga-Saito in \cite{hs11}
about the restriction of representations of $p$-adic groups.
Further, similar arguments are also reviewed for representations of finite groups, based on Clifford theory (cf., see \cite{jamesliebeck01}). 
\subsection{Restriction for $p$-adic groups}  \label{restriction on p-adic}
We continue with connected reductive algebraic groups $\bG$ and $\tbG$ over $F$ satisfying the condition \eqref{cond on G}.
Given $\sigma \in \Irr(G),$ there exists $\ts \in \Irr(\tG)$ such that 
\[
\sigma \subset {\Res}_{G}^{\tG}(\ts),
\]
due to \cite[Lemma 2.3]{gk82} and \cite[Proposition 2.2]{tad92}. 
We use both $\Pi_{\sigma}(G)$ and $\Pi_{\ts}(G)$ for the set of equivalence classes of all irreducible constituents of ${\Res}_{G}^{\tG}(\ts).$ 
It turns out that $\Pi_{\sigma}(G)$ is finite and independent of the choice of the lifting $\ts \in \Irr(\tG)$ (see \cite[Lemma 2.1]{gk82} and \cite[Proposition 2.4 \& Corollary 2.5]{tad92}). 
Further, for any irreducible constituents $\si_1$ and  $\si_2$ in ${\Res}_{G}^{\tG}(\ts),$ it is clear that $\Pi_{\si_1}(G) = \Pi_{\si_2}(G).$ 

Given $\ts_1$, $\ts_2 \in \Irr(\tG),$ we recall the following equivalent statements from \cite[Lemma 2.4]{gk82} and \cite[Corollary 2.5]{tad92}:
\begin{equation} \label{pro for lifting}
\exists \chi \in (\tG / G)^\vee ~~\text{such that}~~ \ts_1 \s \ts_2 \chi ~~ \Leftrightarrow ~~  \Pi_{\ts_1}(G) \cap \Pi_{\ts_2}(G) \neq \emptyset ~~ \Leftrightarrow ~~ \Pi_{\ts_1}(G) = \Pi_{\ts_2}(G),
\end{equation}
where $(\tG / G)^\vee = \Hom(\tG / G, \CC^{\times})$ of all continuous characters of $\tG / G,$ which are considered as continuous 1-dimensional representation of $\tG$ that are trivial on its subgroup $G.$
Since ${\Res}_{G}^{\tG}(\ts)$ is completely reducible due to \cite[Lemma 2.1]{gk82} and \cite[Lemma 2.1]{tad92}, we have the decomposition
\begin{equation} \label{decomp of Res}
{\Res}_{G}^{\tG}(\ts) = \bigoplus _{\tau \in \Pi_{\sigma}(G)} \langle \tau, \ts \rangle_{G} \cdot \tau
\end{equation}
(cf. \cite[Chapter 2]{hs11}).
We note that the multiplicity $\langle \tau, \ts \rangle_{G}$ has the common value over $\tau \in \Pi_{\sigma}(G)$ \cite[Lemma 2.1(b)]{gk82}, that is, 
$
\langle \tau_1, \ts \rangle_{G} = \langle \tau_2, \ts \rangle_{G}
$
for any $\tau_1,~ \tau_2 \in \Pi_{\sigma}(G)$. We define
\begin{equation} \label{X(sigma)}
I(\ts):= \{ \chi \in (\tG/G)^\vee : \ts \s \ts \chi \}.
\end{equation}
Then we have the following equalities.
\begin{pro} \label{pro for multi}
With the above notation, we have 
\[
|I(\ts)| 
= 
{\dim}_{\CC} \; {\Hom}_{G}({\Res}_{G}^{\tG} (\ts), {\Res}_{G}^{\tG} (\ts))
= 
|\Pi_{\sigma}(G)| \cdot \langle \sigma, \ts \rangle_{G}^2.
\]
\end{pro}
\begin{proof}
The first equality is a consequence of \cite[Proposition 2.4]{tad92}. 
The second equality follows from the decomposition \eqref{decomp of Res} and Schur's lemma. Indeed, for any $\tau_1, \tau_2 \subset {\Res}_{G}^{\tG} (\ts),$ we have 
$
{\Hom}_{G}(\tau_1, \tau_2) = 0
$
unless $\tau_1 \s \tau_2,$ in which case ${\Hom}_{G}(\tau_1, \tau_2) \s \CC.$ Thus, the proof is complete.
\end{proof}
We define the stabilizer of $\sigma$ in $\tG$  
\[
\tG_{\sigma}:= \{ g \in \tG : {^g}{\sigma} \s \sigma
\}.
\]
It is known \cite[Corollary 2.3]{tad92} that $\tG_{\sigma}$ is an open normal subgroup of $\tG$ of finite index and satisfies
\begin{equation}  \label{normal subgroups}
Z({\tG}) \cdot G \subseteq \tG_{\sigma} \subseteq \tG.
\end{equation}
Since $\tbG$ and $\bG$ share the same derived group by the condition \eqref{cond on G}, we also note that $Z({\tG}) \cdot G$ is an open normal subgroup of $\tG$ of finite index.
\begin{pro}  (\cite[Lemma 2.1(c)]{gk82}) \label{simply transitive on the set}
The quotient $\tG/\tG_{\sigma}$ acts by conjugation on the set $\Pi_{\sigma}(G)$ simply and transitively. In particular, the set $\Pi_{\sigma}(G)$ is in one-to-one correspondence with the quotient $\tG/\tG_{\sigma}.$ 
\qed
\end{pro}

We review some results in \cite[Chapter 2]{hs11} (see also Section \ref{section for SL}).
Given $\ts \in \Irr(\tG)$ 
and $\chi \in I(\ts),$ 
there is a non-zero endomorphism $I_{\chi} \in {\Aut}_{\CC}(V_{\ts})$ such that $I_{\chi} \circ (\ts \chi)  = \ts \circ I_{\chi}.$
The scalar endomorphism $\tilde v \mapsto z \cdot \tilde v$ for $v \in V_{\ts}$ and $z \in \CC^{\times}$ is defined and denoted by $z \cdot \id_{V_{\ts}}.$ 
Thus, the subgroup consisting of $z \cdot \id_{V_{\ts}}$ in ${\Aut}_{\CC}(V_{\ts})$ is identified with $\CC^\times.$
Define $\mcA(\ts)$ as the subgroup of ${\Aut}_{\CC}(V_{\ts})$ generated by $\{I_{\chi} : \chi \in I(\ts) \}$ and $\CC^\times.$ 
Then the map $I_{\chi} \mapsto \chi$ induces the following exact sequence
\begin{equation} \label{exact for CC1}
1 \longrightarrow \CC^\times \longrightarrow \mcA(\ts) \longrightarrow I(\ts) \longrightarrow 1.
\end{equation}  
Following \cite[p.11]{hs11}, we denote by $\Irr(\mcA(\ts), \id)$ the set of isomorphism classes of irreducible representations of the group $\mcA(\ts)$ such that $z\cdot \id_{V_{\ts}} \in \CC^\times$ acts as the scalar $z.$ 
Recall the following isomorphism from \cite[Corollary 2.10]{hs11}
\begin{equation}  \label{useful decomp}
V_{\ts} ~ ~ \s \bigoplus_{\xi \in \Irr(\mcA(\ts), \id)} \xi \boxtimes  \si_{\xi}
\end{equation}
as representations of the direct product $\mcA(\ts) \times G.$ 
There is thus a one-to-one correspondence
\begin{equation}  \label{bij 333}
\Irr(\mcA(\ts), \id) \s \Pi_{\ts}(G),
\end{equation}
sending $\xi \mapsto  \si_{\xi}.$ 
By $\xi_{\si}$ we denote the inverse of $\sigma$ via the correspondence \eqref{bij 333}.
The isomorphism \eqref{useful decomp} also yields
\begin{equation}  \label{dim=multi}
\langle \sigma, \ts \rangle_{G} = \dim \xi_{\si}.
\end{equation}
\begin{rem} \label{dim1=dim2}
Given $\ts \in \Irr_{\disc}(\tG),$ since the multiplicity is common, the equality \eqref{dim=multi} implies that 
$
{\dim}\xi_{\si_1} = {\dim}\xi_{\si_2}
$  
for any $\si_1, \si_2 \in \Pi_{\ts}(G).$
\end{rem}

\subsection{Restriction for finite groups}  \label{restriction on finite groups}
Given two finite groups $H$ and  $\tH$ satisfying the property that $H$ is a normal subgroup of $\tH$ with abelian quotient, we shall review some results on the restriction of representations from $\tH$ to $H$ along with Clifford theory. We mainly refer to \cite[(20.7) \& (20.8)]{jamesliebeck01}. 
In principle, all the arguments in the restriction for $p$-adic groups in Section  \ref{restriction on p-adic} hold for finite groups as follows.
Given $\delta \in \Irr(H),$ there exists $\tdel \in \Irr(\tH)$ such that 
\[
\delta \subset {\Res}_{H}^{\tH}(\tdel),
\]
and we employ both $\Pi_{\delta}(H)$ and $\Pi_{\tdel}(H)$ for the set of equivalence classes of all irreducible constituents of ${\Res}_{H}^{\tH}(\tdel).$ 
Further, we have $\Pi_{\delta_1}(H) = \Pi_{\delta_2}(H)$ for any irreducible constituents $\delta_1$ and  $\delta_2$ in ${\Res}_{H}^{\tH}(\tdel),$ 
and the decomposition
 \begin{equation*} \label{decomp of Res for finite}
{\Res}_{H}^{\tH}(\tdel) = \bigoplus _{\theta \in \Pi_{\widetilde\delta}(H)} \langle \theta, \tdel \rangle_{H} \cdot \theta.
\end{equation*}
The multiplicity $\langle \theta, \tdel \rangle_{H}$ here is also common over $\theta \in \Pi_{\widetilde\delta}(H).$ 

We define
\begin{equation*} \label{X(delta)}
I(\tdel):= \{ \eta \in (\tH/H)^\vee : \tdel \s \tdel \eta \},
\end{equation*}
where $(\tH/H)^\vee = \Hom(\tH/H, \CC^{\times})$ of all characters of $\tH/H,$ which are considered as 1-dimensional representation of $\tH$ that are trivial on its subgroup $H.$
Choose $\delta \in \Pi_{\widetilde\delta}(H),$ and define the stabilizer of $\widetilde\delta$ in $\tH$  
\[
\tH_{\delta}:= \{h \in \tH : {^h}{\delta} \s \delta
\}.
\]
It is obvious that $\tH_{\delta}$ satisfies $Z({\tH}) \cdot H \subseteq \tH_{\delta} \subseteq \tH$ and is a normal subgroup of $\tH.$
Further, the quotient $\tH/\tH_{\delta}$ acts by conjugation on the set $\Pi_{\widetilde\delta}(H)$ simply and transitively.
We also have the equality $|I(\tdel)| = |\Pi_{\widetilde\delta}(H)| \cdot \langle \delta, \tdel \rangle_{H}^2,$ since $\tH/H$ is abelian.

\section{Multiplicity in restriction for tempered representations}  \label{section of main results}
We continue with the notation in Sections \ref{backgrounds} and \ref{restriction on p-adic and finite}.
Given two connected reductive algebraic groups $\bG$ and $\tbG$ over $F$ satisfying \eqref{cond on G}, under working hypotheses about the tempered local Langlands conjecture and the internal structure of tempered $L$-packets, 
we construct two crucial bijections in Section \ref{main}.
We then apply these bijections to establish an equality of multiplicities in $p$-adic groups and their $\cS$-groups, 
and formulate the multiplicity in the restriction in $p$-adic groups by means of dimensions of irreducible representations of their $\cS$-groups in Section \ref{both multi}.
\subsection{Two bijections} \label{main}
This section is devoted to constructing two bijections between finite abelian groups occurring in $p$-adic groups and their $\cS$-groups: one between two quotients in $\tG$ and a character group under Working Hypothesis \ref{wh temp llc}, \ref{shahidi's conj}, \ref{last hyp}, and \ref{llast hyp} (see Theorem \ref{thm 1}); and the other one between two quotients in $\cS$-groups and a cohomological group under Working Hypothesis \ref{wh temp llc}, \ref{shahidi's conj}, \ref{last hyp}, and \ref{really llast hyp} (see Theorem \ref{thm 2}).

To begin with, the local Langlands conjecture predicts that there is a surjective, finite-to-one map from $\Irr(G)$ to $\Phi(G),$ which is expected to satisfy various natural properties.
The map preserves $\gamma$-factors, $L$-factors, and $\epsilon$-factors, if they are available in both sides (cf. \cite{ht01, he00}). 
Furthermore, through the map, the subsets $\Phi_{\disc}(G)$ and $\Phi_{\temp}(G)$ in $\Irr(G)$ are characterized by the subsets $\Irr_{\disc}(G)$ and $\Irr_{\temp}(G)$ in $\Phi(G),$ respectively. 

The {first} working hypothesis is the local Langlands conjecture for tempered representations of $p$-adic groups as follows.
\begin{wh} \label{wh temp llc}
The local Langlands conjecture is true for irreducible tempered representations of $G,$ i.e., there is a surjective, finite-to-one map
\begin{equation} \label{llc bij}
{\L_G}: {\Irr}_{\temp}(G) \longrightarrow {\Phi}_{\temp}(G).
\end{equation}
\end{wh}
\begin{rem}
We shall need the temperedness condition so as to state Shahidi's conjecture (Working Hypothesis \ref{shahidi's conj} below) which will be related to other working hypotheses and used in Section \ref{section of main results}.
We may further note that the study of $\Irr(G)$ can be reduced to its subset $\Irr_{\temp}(G),$ due to the Langlands classification theorem (see \cite[Section 3]{konno03} for a proof and survey for the theorem).
Furthermore, we shall refer the reader to \cite[Section 10]{bo79} for several properties that ${\L_G}$ is expected to satisfy.
\end{rem}

Given a tempered $L$-parameter $\vp \in \Phi_{\temp}(G),$ each $L$-packet $\Pi_{\vp}(G):=\L_G^{-1}(\vp)$ is conjectured to be parameterized in terms of the group of connected components of a certain centralizer in the $L$-group.
We recall such conjectural parametrization in detail mainly based on \cite[Section 4.2]{kalsim} and \cite[Section 5.4]{kalrigd15} as follows.

Let $\bG^*$ be a quasi-split form of $\bG$ over $F$ ($\bG^*$ is $\bG$ itself when $\bG$ is quasi-split over $F$) i.e., $\bG^*$ is a connected reductive quasi-split group over $F$ and there is an $\bar{F}$-isomorphism $\xi: \bG \overset{\sim}{\rightarrow} \bG^*$ such that
 $\xi \circ \tau(\xi)^{-1}$ is an inner automorphism ($g \mapsto xgx^{-1}$) defined over $\bar{F}$ for all $\tau \in \Gal (\bar{F} / F)$ (see \cite[2.4(3)]{bo79} or \cite[p.280]{kot97}). 
 Such $\xi$ is called an inner twist.
Fix a Whittaker datum $\mm.$ Here $\mm$ stands for a $G^*$-conjugacy class of pairs $(\bB^*, \psi^*),$ where $\bB=\bT^*\bU^*$ is a Borel subgroup of $\bG^*$ defined over $F$ with maximal $F$-torus $\bT^*$ and unipotent radical $\bU^*,$ and $\psi^*$ is a generic character of $U^*.$
There is a one-to-one correspondence between the set of $T^*$-orbits of the generic characters and the quotient $\bG^*_{\ad}(F)/\imag\big(\bG^*(F) \rightarrow \bG^*_{\ad}(F)\big)$ that is embedded into $H^1(F, Z(\bG^*)).$ 
Given a Whittaker datum $\mm=(\bB^*, \psi^*),$ an irreducible admissible representation $\pi \in \Irr(G^*)$ is called $\mm$-generic if there exists a non-zero $\psi^*$-generic Whittaker functional for $\pi,$ equivalently if $\Hom_{U^*}(\pi, \psi^*) \neq 0$ (for the details, see \cite[Chapter 3]{sh10}). 

The {second} working hypothesis is a strong form of Shahidi's conjecture in \cite[Section 9]{sh90} as follows.

\begin{wh} \label{shahidi's conj}
Assume Working Hypothesis \ref{wh temp llc} is true for $\bG^*.$
Let an $L$-parameter $\vp \in \Phi_{\temp}(G^*)$ be given.  For each Whittaker datum $\mm,$ there is a unique $\mm$-generic member in the $L$-packet $\Pi_{\vp}(G^*)$ associated to $\vp.$
\end{wh}

The next hypothesis is related to the internal structure of $L$-packets. We shall introduce two relevant formulations in order.
The first formulation suggested by Arthur \cite{art06} utilizes a finite group $\cS_{\vp, \scn},$ which is described as follows. By $\widehat{\bG}_{\scn}$ we denote the simply connected cover of the derived group $\widehat{\bG}_{\der}$ of $\widehat{\bG}.$ 
Let $\widehat{\bG}_{\ad}$ be the adjoint group $\widehat{\bG}/Z(\widehat{\bG})$ of $\widehat \bG.$
Then the quotient
\[
S_{\vp}(\widehat{\bG}):=C_{\vp}(\widehat \bG) / Z(\widehat{\bG})^{\Gamma},
\]
is considered as a subgroup in $\widehat{\bG}_{\ad}.$
Via the isogeny $\widehat{\bG}_{\scn} \twoheadrightarrow \widehat{\bG}_{\ad},$ we get the full pre-image of $S_{\vp}(\widehat{\bG}),$ which is denoted by $S_{\vp, \scn}(\widehat{\bG}).$ 
We then have an exact sequence
\begin{equation} \label{exact isogeny}
1 \longrightarrow Z(\widehat \bG_{\scn}) \longrightarrow S_{\vp, \scn}(\widehat{\bG}) \longrightarrow S_{\vp}(\widehat{\bG}) \longrightarrow 1.
\end{equation}
We put 
\begin{align*}
\cS_{\vp}(\widehat{\bG}) &:= \pi_0(S_{\vp}(\widehat{\bG})), \\
\cS_{\vp, \scn}(\widehat{\bG}) & := \pi_0(S_{\vp, \scn}(\widehat{\bG})), \\
\widehat Z_{\vp, \scn}(\bG) & := Z(\widehat \bG_{\scn}) / (Z(\widehat \bG_{\scn}) \cap S_{\vp, \scn}(\widehat{\bG})^{\circ}),
\end{align*}
and have an exact sequence
\begin{equation} \label{central ext}
1 \longrightarrow \widehat Z_{\vp, \scn}(\bG)  \longrightarrow \cS_{\vp, \scn}(\widehat{\bG}) \longrightarrow \cS_{\vp}(\widehat{\bG}) \longrightarrow 1.
\end{equation}

Consider again an inner twist $\xi: \bG^* \to \bG$ with $\bG^*$ quasi-split form of $\bG.$
Choose a character $\zeta_{\bG}$ of $Z(\widehat \bG_{\scn})$ whose restriction to $Z(\widehat \bG_{\scn})^{\Gamma}$ corresponds to 
the class of the $F$-inner form $\bG$ of $\bG^*$ via the Kottwitz isomorphism \cite[Theorem 1.2]{kot86} (note here that $\widehat \bG = \widehat \bG^*$).  
Following \cite[p.549-550]{art12}, we denote by $\Irr(\cS_{\vp, \scn}(\widehat{\bG}), \zeta_{\bG})$ the set of irreducible representations of $\cS_{\vp, \scn}(\widehat{\bG})$ that is equivariant under the pullback of $\zeta_{\bG}$ to $\widehat Z_{\vp, \scn}(\bG)$ (cf., \cite[Section 4.6]{kalrigd15-global}, \cite[p.209]{art06}).

Given a \textit{tempered} $L$-parameter $\vp$ for $G$ and a \textit{tempered} $L$-packet $\Pi_{\vp}(G) $ associated to $\vp,$ 
it is conjectured \cite[Section 3]{art06} that
there is a one-to-one correspondence 
\begin{equation} \label{internal st}
\Pi_{\vp}(G) \overset{1-1}{\longleftrightarrow} \Irr(\cS_{\vp, \scn}(\widehat{\bG}), \zeta_{\bG}).
\end{equation}
This conjecture provides the internal structure of $\Pi_{\vp}(G).$
In particular, when $\bG^*=\bG,$  the character $\zeta_{\bG}$ equals the trivial character $\mathbbm{1}$ and it is clear that 
$\Irr(\cS_{\vp, \scn}(\widehat{\bG}), \mathbbm{1}) = \Irr(\cS_{\vp}(\widehat{\bG})).$

The other formulation suggested by Kaletha \cite{kalsim} utilizes a finite group $\cS_{\vp}^+,$ which is described as follows. 
Fix a finite subgroup $\bZ \subset Z(\bG^*)$ defined over $F$ and set $\bar \bG^* := \bG^*/\bZ.$ We then have an isogeny $\widehat{\bar \bG^*} \twoheadrightarrow \widehat{\bG^*}$ dual to  the isogeny $\bG^* \twoheadrightarrow \bar \bG^*.$
For our purpose, we only consider the case that $\bZ = Z(\bG_{\der}^*)$ from now on, though the following notions are available in the general case of $\bZ.$
According to \cite[Section 3]{kalsim} and \cite[Section 5]{kalrigd15}, the new cohomology set, denoted by
$H^1(u \rightarrow W, \bZ \rightarrow \bG^*),$ has the property that there is a surjective map
\begin{equation} \label{surj on cohos}
H^1(u \rightarrow W, \bZ \rightarrow \bG^*) \twoheadrightarrow H^1(F, \bG^*/Z(\bG^*)).
\end{equation}
We note from \cite[Corollary 3.8]{kalsim} that the map \eqref{surj on cohos} is bijective when $\bG^*$ is split.
Following \cite[Section 5.1]{kalrigd15} and \cite[Section 4]{kalsim}, a rigid inner twist $(\xi', z')$ is defined as a pair consisting of an inner twist $\xi':\bG' \rightarrow \bG^*$ and an element $z' \in Z^1(u \rightarrow W, \bZ \rightarrow \bG^*)$ such that for all $\tau \in \Gamma,$ we have $\xi'^{-1}\tau(\xi')=Ad(\bar z'(\tau)),$ where $\bar z' \in Z^1(F, \bG^*/Z(\bG^*))$ is the image of $z'$ under the map \eqref{surj on cohos}.

\begin{rem}
Since the map \eqref{surj on cohos} is surjective, we have $z' \in Z^1(u \rightarrow W, \bZ \rightarrow \bG^*)$ for every inner twist $\xi':\bG' \rightarrow \bG^*$ such that $(\xi', z')$ is a rigid inner twist (see \cite[Section 5.4]{kalrigd15}). Thus, since $\bG^*$ is a quasi-split form of $\bG$ over $F,$ there is a corresponding element in $Z^1(u \rightarrow W, \bZ \rightarrow \bG^*)$ to the given connected reductive group $\bG$ over $F$ and inner twist $\xi: \bG \rightarrow \bG^*.$ 
\end{rem}

Let $\vp \in \Phi_{\temp}(G^*)$ be a tempered $L$-parameter. 
For each rigid inner twist $(\xi', z'): \bG' \rightarrow \bG^*$ with $z' \in Z^1(u \rightarrow W, \bZ \rightarrow \bG^*),$ we let $\Pi_\vp((\xi', z'))$ be the $L$-packet $\Pi_\vp(G')$ under Working Hypothesis \ref{wh temp llc}.
We note that $\Pi_\vp((id, 1))=\Pi_\vp(G^*).$ Also, due to the relevance condition, we have $\Phi_{\temp}(G') \subset \Phi_{\temp}(G^*)$ and $\Pi_\vp(G')$ may be empty if $\vp$ is not relevant to $\bG',$ i.e., $\vp \notin \Phi_{\temp}(G')$ (see \cite[Section 5.4]{kalrigd15}).
Recalling the centralizer $C_\vp$ in $\widehat{\bG^*}$ of the image of $\vp,$ we denote by $S_\vp^+$ the preimage of $C_\vp$ in $\widehat{\bar \bG^*}$ under the isogeny $\widehat{\bar \bG^*} \twoheadrightarrow \widehat{\bG^*}.$
We let denote $\cS^+_{\vp}$ the group $\pi_0(S_{\vp}^+)$ of connected components of $S^+_{\vp},$ which is finite, and let $\Irr(\cS^+_{\vp})$ denote the set of irreducible representations of $\pi_0(S_{\vp}^+).$ 
Given a tempered $L$-parameter $\vp \in \Phi_{\temp}(G^*)$ and a Whittaker datum $\mm,$ there exists a bijective map
\begin{equation} \label{for 3rd wh}
\iota_{\mm} : \bigsqcup_{(\xi', z')} \Pi_\vp((\xi', z')) \longrightarrow \Irr(\cS_\vp^+),
\end{equation}
such that the image of the unique $\mm$-generic member of $\Pi_\vp((id, 1))$ is the trivial representation of $\cS_\vp^+.$
Let a rigid inner twist $(\xi', z'): \bG' \rightarrow \bG^*$ be given. Recalling the notation in \cite[Section 4.1]{kalsim}, we consider a refined endoscopic quadruple $\dot\ee=(G^\ee, \mathcal G^\ee, s^{\dot\ee}, \eta^\ee)$ for $\bG^*.$
For any $\Delta[\mm, \dot\ee, z']$-matching functions $f^{\dot\ee} \in C_c^\infty(G^{\ee})$ and $f'\in C_c^\infty(G'),$ we have the equality 
\[
\Theta^1_{\vp^\ee}(f^{\dot\ee})=e(G') \sum_{\pi \in \Pi_{\vp}((\xi', z'))} \langle \pi', s^{\dot\ee} \rangle \Theta_{\pi'}(f'), 
\]
where $\vp^\ee \in \Phi_{\temp}(G^\ee)$ is such that $\vp = {^L}\eta^\ee \circ \vp^\ee$ and $\langle \pi', - \rangle = tr(\iota_\mm(\pi') (-)).$

It is important to discuss the relationship between two formulations: one is in terms of $\cS_\vp^+$ and the other is $\cS_{\vp, \scn}.$ Following \cite[Section 4.6]{kalsim} and \cite[Section 4.6]{kalrigd15-global}, the character $\zeta_{\bG}$ of $Z(\widehat \bG_{\scn})$ that is described below Equation \eqref{central ext}, 
determines a cohomology class in $H^1(u \rightarrow W, Z(\bG^*_{\scn}) \rightarrow \bG^*_{\scn})$ (due to being in the non-archimedean case) whose image in $H^1(u \rightarrow W, Z(\bG^*_{\der}) \rightarrow \bG^*)$ gives a character, denoted by $\zeta_{\bG}^+,$ of $\pi_0(Z(\widehat{\bar G^*})^+).$ Here $Z(\widehat{\bar G^*})^+$ is the preimage in $\widehat{\bar G^*}$ under the isogeny $\widehat{\bar \bG^*} \twoheadrightarrow \widehat{\bG^*}$ of the diagonalizable group $Z(\widehat{G^*})^{\Gamma}.$ It is also noted that the character $\zeta_{\bG}^+$ is the pull-back of $\zeta_{\bG}$ under the map $\pi_0(Z(\widehat{\bar G^*})^+) \rightarrow Z(\widehat \bG_{\scn}),$ since we work with the case $\bZ = Z(\bG_{\der}^*).$
We then have the following bijection
\begin{equation} \label{useful 1-1}
\Irr(\cS_\vp^+, \zeta_{\bG}^+) \overset{1-1}{\longleftrightarrow} \Irr(\cS_{\vp, \scn}(\widehat{\bG}), \zeta_{\bG}).
\end{equation}
We refer the reader to \cite[p.244]{kalsim} and \cite[(4.6)]{kalrigd15-global} for further differences between these two groups.

\begin{rem} \label{rem for int structure}
The bijection \eqref{internal st} depends on a certain choice of data. For quasi-split $\bG$ with a tempered $L$-parameter $\vp,$ the choice of Whittaker datum yields a unique bijection in \eqref{internal st} so that a unique generic representation in $\Pi_{\vp}(G)$ with respect to the Whittaker datum corresponds to the trivial representation of $\Irr(\cS_{\vp}(\widehat{\bG})).$ 
For non-quasi-split $\bG$ with a tempered $L$-parameter $\vp,$ despite the absence of the Whittaker datum, Kaletha carried out the dependence and made the bijection \eqref{internal st} canonical by means of a certain cohomological group related to Galois gerbes in \cite{kalrigd15}. To be precise, we consider an inner twist $\xi: \bG^* \to \bG$ with $\bG^*$ quasi-split form of $\bG,$ and an element $z$ of $Z^1(u \rightarrow W, Z(\bG^*_{\scn}) \rightarrow \bG^*_{\scn})$ that lifts the element of $Z^1(F, \bG^*_{\ad})$ given by $\xi^{-1}\tau(\xi)$ for $\tau \in \Gamma$  (see \cite[Section 4]{kalrigd15-global} for more details). Then, fixing a Whittaker datum for $\bG^*,$ the datum of a pair $(\xi,z)$ turns out to make the bijection \eqref{internal st} \textbf{canonical}. In particular, if $\bG$ is quasi-split and $\bG^*=\bG,$ then the pair $(\xi,z)$ can be taken to be the trivial pair $(id,1),$ so that the fixed Whittaker datum for $\bG$ alone is sufficient to make the bijection \eqref{internal st} canonical, as explained in the case of quasi-split $\bG$ above. 
The author would like to thank Kaletha for valuable communications on this remark.
\end{rem}

The {third} working hypothesis is the internal structure of tempered $L$-packets as follows.
\begin{wh} \label{last hyp}
Let $\bG$ be a connected reductive algebraic group over $F.$ Consider an inner twist $\xi: \bG^* \to \bG$ with $\bG^*$ quasi-split form of $\bG$ and an element $z$ of $Z^1(u \rightarrow W, Z(\bG^*_{\scn}) \rightarrow \bG^*_{\scn})$ that lifts the element of $Z^1(F, \bG^*_{\ad})$ given by $\xi^{-1}\sigma(\xi).$ 
Assume Working Hypotheses \ref{wh temp llc} and \ref{shahidi's conj}. Given a tempered $L$-parameter $\vp \in \Phi_{\temp}(G^*)$ and a Whittaker datum $\mm,$ the bijection \eqref{for 3rd wh} holds for  $\bG,$ so that the bijection \eqref{internal st} is canonical, due to Remark \ref{rem for int structure} and the bijection \eqref{useful 1-1}.
\end{wh}
\begin{rem} \label{after 3rd wh}
We shall use $\cS_{\vp, \scn}$ in this paper rather than $\cS_\vp^+,$ under Working Hypothesis \ref{last hyp}, due to the complexity of the group structure of $\cS_\vp^+.$ 
\end{rem}

Now we construct two bijections in Theorems \ref{thm 1} and \ref{thm 2} below.
From now on, we assume that Working Hypotheses \ref{wh temp llc}, \ref{shahidi's conj}, and \ref{last hyp} are valid for $G$ and $\tG.$
Due to \cite[Section 1]{kot84}, the exact sequence given by the condition \eqref{cond on G}
\begin{equation} \label{before an exact before a def}
1 \longrightarrow \bG \longrightarrow \tbG \longrightarrow \tbG/\bG \longrightarrow 1
\end{equation}
yields
\begin{equation} \label{an exact before a def}
1 \longrightarrow \widehat{\tbG/\bG} \longrightarrow \widehat \tbG \overset{pr}{\longrightarrow} \widehat \bG \longrightarrow 1
\end{equation}
(see also \cite[Remark 2.4]{chgo12}). 
Thus, the kernel $\widehat S$ in \eqref{Labesse situation} is now $\widehat {(\tbG/\bG)}.$ 

Let $\vp \in \Phi_{\temp}(G)$ be given, and choose a lifting $\tvp \in \Phi_{\temp}(\tG)$ of $\vp,$ namely $\vp=pr \circ \tvp,$ as in Theorem \ref{thm by Labesse}. 
The surjective map $pr: \widehat{\tbG} \rightarrow \widehat{\bG}$  in \eqref{an exact before a def}  induces a natural map between centralizers
\begin{equation*} \label{prC}
pr_C: C_{\tvp}(\widehat \tbG) \longrightarrow C_{\vp}(\widehat \bG).
\end{equation*}
The kernel of $pr_C$ then equals the subgroup $\big(\widehat{\tbG/\bG}\big)^\Gamma$ of $\Gamma$-invariants in the kernel of $pr$ in \eqref{an exact before a def}. 
Applying \eqref{before an exact before a def} to \cite[(1.8.1)]{kot84} and taking $\Gamma$-invariants, we have
\begin{equation} \label{hat exact}
1 \longrightarrow \big(\widehat{\tbG/\bG}\big)^\Gamma \longrightarrow Z(\widehat \tbG)^\Gamma  \longrightarrow Z(\widehat \bG)^\Gamma \longrightarrow H^1(F, \widehat{\tbG/\bG}) \longrightarrow H^1(F, Z(\widehat \tbG))  \longrightarrow \cdots.
\end{equation}

As noted in Section \ref{notation},  the restriction map $H^1(F, \widehat{\tbG/\bG}) \rightarrow H^1(W_F, \widehat{\tbG/\bG})$ yields the image, $\imag\big(Z(\widehat \bG)^\Gamma \rightarrow H^1(W_F, \widehat{\tbG/\bG})\big),$ of $Z(\widehat \bG)^\Gamma$ in $H^1(W_F, \widehat{\tbG/\bG}),$
and define the following quotient
\begin{equation} \label{def of X}
X(\tvp) := \{ a \in H^1(W_F, \widehat {(\tbG/\bG)}) : a \tvp \s \tvp \mbox{ in } \widehat{\tbG} \}\big/\imag\big(Z(\widehat \bG)^\Gamma \rightarrow H^1(W_F, \widehat{\tbG/\bG})\big).
\end{equation}
\begin{rem} \label{rem for X}
The set $\{ a \in H^1(W_F, \widehat {(\tbG/\bG)}) : a \tvp \s \tvp \mbox{ in } \widehat{\tbG} \}$ is a finite abelian group (cf., \cite[p.74]{chaoli}), so is $X(\tvp).$ 
Furthermore, in \cite[Lemma 5.3.4]{chaoli}, since $Z(\widehat \bG^\sharp)^\Gamma$ is the identity of the quotient group $S_{\phi^\sharp, \ad}$ therein (note that their notation $\bG, \phi, S_{\phi, \ad}, \bG^\sharp, \phi^{\sharp},$ and $S_{\phi^\sharp, \ad}$ respectively correspond to our notation $\tbG, \tvp, S_{\tvp}(\widehat{\tbG}), \bG, \vp,$ and $S_{\vp}(\widehat{\bG})$), 
it seems that their group $X^G(\phi)$ needs to take quotient by $\imag\big(Z(\widehat \bG^\sharp)^\Gamma \rightarrow H^1(W_F, \widehat{Z}^\sharp)\big),$ though their lemma remains unchanged, as suggested from \eqref{commutative diagrams for main} below (see also \cite[(3.4) and (3.5)]{xu15}). 
We further refer the reader to \cite[Theorem 4.3]{gk82} and \cite[Proposition 2.9]{gtsp10} where some other groups than $X(\tvp)$ were discussed in different situation.
\end{rem}
In the following lemma, we slightly enhance \cite[Lemma 5.3.4]{chaoli} by adding the injective argument.
\begin{lm}  (\cite[Lemma 5.3.4]{chaoli}) \label{lemma of Chaoli}
With the above notation, given $\vp \in \Phi_{\temp}(G)$ and $\tvp \in \Phi_{\temp}(\tG)$ with $\vp = \tvp \circ pr$ as in Theorem \ref{thm by Labesse},  we have an exact sequence of finite groups
\[
1 \longrightarrow \cS_{\tvp}(\widehat{\tbG}) \longrightarrow \cS_{\vp}(\widehat{\bG}) \longrightarrow X(\tvp) \longrightarrow 1.
\]
\end{lm}
\begin{proof}
From \cite[Lemma 5.3.4]{chaoli} and Remark \ref{rem for X}, we have an exact sequence
\begin{equation} \label{Chao-Li exact}
S_{\tvp}(\widehat{\tbG}) \longrightarrow S_{\vp}(\widehat{\bG}) \longrightarrow X(\tvp) \longrightarrow 1.
\end{equation}
It now remains to show that $\ker(\cS_{\tvp}(\widehat{\tbG}) \rightarrow \cS_{\vp}(\widehat{\bG}))$ is trivial. 
Combining \eqref{hat exact} and \eqref{Chao-Li exact},
we have the following commutative diagram:
\begin{equation} \label{commutative diagrams for main}
\begin{CD}
@. 1 @. 1 @. 1 \\
@. @VVV @VVV @VVV @.\\
1 @>>> \big(\widehat{\tbG/\bG}\big)^\Gamma @>>> Z(\widehat \tbG)^\Gamma @>>> Z(\widehat \bG)^\Gamma @>>> H^1(F, \widehat{\tbG/\bG}) \\
@. @| @VV{\cap}V @VV{\cap}V @|\\
1 @>>> \big(\widehat{\tbG/\bG}\big)^\Gamma @>>> C_{\tvp}(\widehat \tbG) @>>> C_{\vp}(\widehat \bG)@>>> H^1(F, \widehat{\tbG/\bG}) \\
@. @VVV @VVV @VVV @. \\
@. 1@>>>S_{\tvp}(\widehat{\tbG}) @>>> S_{\vp}(\widehat{\bG}) @>>> X(\tvp) \\
@. @. @VVV @VVV @.\\
@. @.  1 @. 1 
\end{CD}
\end{equation}
The exactness on the bottom of \eqref{commutative diagrams for main}:
\[
1 \longrightarrow S_{\tvp}(\widehat{\tbG}) \longrightarrow S_{\vp}(\widehat{\bG})
\]
and the connectedness argument (cf., \cite[Lemma 4.8]{ch15}) thus verify that
\[
\ker\Big(\cS_{\tvp}(\widehat{\tbG}) \rightarrow \cS_{\vp}(\widehat{\bG})\Big) = \{1\}.
\]
Therefore, the proof of Lemma \ref{lemma of Chaoli} is complete.
\end{proof}
\begin{rem} \label{rem for X to H}
The continuous homomorphism $W_F \rightarrow \Gamma$ has dense image (see Section \ref{notation}), the exact sequence \eqref{hat exact} can be replaced by
\begin{equation} \label{exact with W_F}
1 \longrightarrow \big(\widehat{\tbG/\bG}\big)^\Gamma \longrightarrow Z(\widehat \tbG)^\Gamma  \longrightarrow Z(\widehat \bG)^\Gamma \longrightarrow H^1(W_F, \widehat{\tbG/\bG}) \longrightarrow H^1(W_F,  Z(\widehat \tbG))  \longrightarrow \cdots.
\end{equation}
Accordingly, $H^1(F, -)$ in the diagram \ref{commutative diagrams for main} can be replaced by  $H^1(W_F, -)$ (cf., \cite[Section 4]{karpuk13} and \cite[Section 3.2]{xu15}).
\end{rem}

We first consider the case that $\vp$ is elliptic (so is $\tvp$), i.e., $C_{\phi}(\widehat \bG) / Z(\widehat{\bG})^{\Gamma}$ is finite.
Then, 
$S_{\vp}(\widehat{\bG})$ itself equals $\cS_{\vp}(\widehat{\bG})$ and
$\widehat Z_{\vp, \scn}(\widehat{\bG}_{\scn})$ is identical with $Z(\widehat \bG_{\scn}).$
Hence, the exact sequence \eqref{central ext} is equal to \eqref{exact isogeny}. 
Combining Lemma \ref{lemma of Chaoli} and the exact sequence \eqref{exact isogeny} for $\tbG$ and $\bG,$
we have the following commutative diagram:
\begin{equation} \label{commutative diagrams for Sp11}
\begin{CD}
@. @. 1 @. 1 \\
@. @. @VVV @VVV @.\\
1 @>>> Z(\widehat \tbG_{\scn}) @>>> S_{\tvp, \scn}(\widehat{\tbG}) @>>> S_{\tvp} (\widehat{\tbG}) @>>> 1 \\
@. @| @VV{\cap}V @VV{\cap}V @.\\
1 @>>> Z(\widehat \bG_{\scn}) @>>> S_{\vp, \scn}(\widehat{\bG}) @>>> S_{\vp}(\widehat{\bG}) @>>> 1 \\
@. @VVV @VVV @VVV @. \\
@. 1@>>>S_{\vp, \scn}(\widehat{\bG})\big/S_{\tvp, \scn}(\widehat{\tbG}) @>{\s}>> X(\tvp) @>>> 1\\
@. @. @VVV @VVV @.\\
@. @.  1 @. 1
\end{CD}
\end{equation}
The bottom isomorphism comes from the Snake Lemma (cf., \cite[Corollary 6.12]{rotman09}) which is applied to the other two horizontal exact sequences.
We thus have the following exact sequence (the middle vertical one):
\[
1 \longrightarrow S_{\tvp, \scn}(\widehat{\tbG})  \longrightarrow S_{\vp, \scn}(\widehat{\bG})  \longrightarrow X(\tvp) \longrightarrow 1.
\]

For tempered $L$-parameters, from \cite[Lemma 4.8]{ch15}, we have $S_{\tvp, \scn}(\widehat{\tbG})^\circ = S_{\vp, \scn}(\widehat{\bG})^\circ,$ which implies that $\widehat Z_{\vp, \scn}(\tbG) = \widehat Z_{\vp, \scn}(\bG)$ by definition.
Thus, it follows from the same way as in \eqref{commutative diagrams for Sp11} that
\begin{equation} \label{imp exact}
1 \longrightarrow \cS_{\tvp, \scn}(\widehat{\tbG})  \longrightarrow \cS_{\vp, \scn}(\widehat{\bG})  \longrightarrow X(\tvp) \longrightarrow 1.
\end{equation}

Given $\sigma \in \Pi_{\vp}(G),$ we fix a lifting $\ts \in \Pi_{\tvp}(\tG).$  The internal structure \eqref{internal st} for $G$ and $\tG$ provides two finite representations $\rho \in \Irr(\cS_{\vp, \scn}(\widehat \bG), \zeta_{\bG} )$ and $\trho \in  \Irr(\cS_{\tvp, \scn}(\widehat{\tbG}), \zeta_{\tbG} )$ corresponding to $\si$ and $\ts,$ respectively. 
Note that the condition \eqref{cond on G} yields $\zeta_{\bG} =\zeta_{\tbG} .$
We fix $\sigma$ corresponding to $\rho$ and $\ts$ corresponding to $\trho,$ under the conjectural bijection \eqref{internal st} (cf., Remark \ref{rem for int structure}).
For simplicity of notation, we set $\cS_{\vp, \scn}(\widehat \bG)=\cS_{\vp, \scn}$ and $\cS_{\tvp, \scn}(\widehat \tbG)=\cS_{\tvp, \scn}.$

Applying Galois cohomology to \eqref{before an exact before a def}, we have the following exact sequence
\begin{equation} \label{exact of G tG}
1 \rightarrow G \rightarrow \tG \rightarrow (\tbG/\bG)(F) \rightarrow H^1(F, \bG) \rightarrow  H^1(F, \tbG) \rightarrow H^1(F, \tbG/\bG).
\end{equation}
Since the cokernel of $G \rightarrow \tG$ in \eqref{exact of G tG} is embedded into the $F$-points $(\tbG/\bG)(F)$ of $\tbG/\bG,$ we have
\[
 1 \longrightarrow \tG/G \longrightarrow (\tbG/\bG)(F) \longrightarrow (\tbG/\bG)(F)/(\tG/G) \longrightarrow 1.
\]
Taking the dual $(~-~)^\vee = \Hom(~-~, \CC^{\times})$ and using the local Langlands correspondence for $\tbG/\bG$ (cf., \cite{yu09}), we have  
\begin{equation} \label{exact 111}
 1 \longrightarrow \Big( (\tbG/\bG)(F)/(\tG/G) \Big)^\vee \longrightarrow\Big(  (\tbG/\bG)(F) \Big)^\vee \s H^1(W_F, \widehat{\tbG/\bG}) \longrightarrow \Big( \tG/G \Big)^\vee  \longrightarrow 1.
\end{equation}
Furthermore, the exact sequence
\eqref{exact of G tG} and the bijection $H^1(F, \bG) \s \pi_0(Z(\widehat \bG)^{\Gamma})^\vee$ (\cite[Proposition 6.4]{kot84}) yield 
\begin{equation*} \label{exact 222}
 1 \longrightarrow  (\tbG/\bG)(F)/(\tG/G) \longrightarrow \pi_0(Z(\widehat \bG)^{\Gamma})^\vee  \longrightarrow 
  \pi_0(Z(\widehat \bG)^{\Gamma})^\vee / \ker\Big( \pi_0(Z(\widehat \bG)^{\Gamma})^\vee \rightarrow \pi_0(Z(\widehat \tbG)^{\Gamma})^\vee \Big)
\end{equation*}
and we then have a surjective map
\[
\pi_0(Z(\widehat \bG)^{\Gamma})  \twoheadrightarrow \Big((\tbG/\bG)(F)\big/(\tG/G) \Big)^\vee.  
\]
Using \eqref{exact 111} and the surjection $Z(\widehat \bG)^{\Gamma} \twoheadrightarrow \pi_0(Z(\widehat \bG)^{\Gamma}),$ we then have
\begin{equation}  \label{for ker indeed} 
Z(\widehat \bG)^{\Gamma}  \twoheadrightarrow  \ker \Big(  \big((\tbG/\bG)(F)\big)^\vee  \twoheadrightarrow \big(\tG/G\big)^\vee\Big). 
\end{equation} 
From \eqref{exact with W_F} that the set $\imag\big( Z(\widehat \bG)^\Gamma \rightarrow H^1(W_F, \widehat{\tbG/\bG})\big)$ vanishes in $H^1(W_F,  Z(\widehat \tbG)),$ and it induces the trivial character on $\tG$ via the map, $H^1(W_F,  Z(\widehat \tbG)) \rightarrow (\tG)^\vee,$ described in \cite[Section 10.2]{bo79} (see also \cite[Appendix A]{xu15} for quasi-split cases).
From this argument, \eqref{exact 111} and \eqref{for ker indeed}, 
it follows that the image of $Z(\widehat \bG)^\Gamma$ in $H^1(W_F, \widehat {(\tbG/\bG)})  \s \big((\tbG/\bG)(F)\big)^\vee$
equals the kernel of 
the surjective map
\begin{equation} \label{surj}
 \Big((\tbG/\bG)(F)\Big)^\vee  \twoheadrightarrow \big(\tG/G\big)^\vee
\end{equation}
(see also \cite[Section 4]{kal2genric}, \cite[Section 5.5]{kalinv}, and \cite[Section 3]{xu15} for similar arguments in a different setting).
Recalling the definition of $X(\tvp)$ in \eqref{def of X}, we then have
\begin{equation}  \label{for pre hom}
X(\tvp) 
\hookrightarrow \Big(H^1(W_F, \widehat {(\tbG/\bG)}) \big/ \imag\big(\Z(\widehat \bG)^\Gamma \rightarrow H^1(W_F, \widehat{\tbG/\bG})\big) \Big)
 \overset{\s}{\longrightarrow} \big(\tG/G \big)^\vee.
\end{equation}
Combining the isomorphism $X(\tvp) \s  \cS_{\vp, \scn}/\cS_{\tvp, \scn}$ from \eqref{imp exact}, we have 
\begin{equation}  \label{for hom}
\tG/G
\overset{\s}{\longrightarrow} \Big(H^1(W_F, \widehat {(\tbG/\bG)}) \big/ \imag\big(\Z(\widehat \bG)^\Gamma \rightarrow H^1(W_F, \widehat{\tbG/\bG})\big)\Big)^\vee 
\twoheadrightarrow  \Big(\cS_{\vp, \scn}/\cS_{\tvp, \scn} \Big)^\vee.
\end{equation}
Hence, we obtain the following surjective homomorphism 
\begin{equation} \label{pre defined map for thm1}
\Lambda: \tG/G \twoheadrightarrow (\cS_{\vp, \scn}/\cS_{\tvp, \scn})^\vee.
\end{equation}

The {fourth} working hypothesis proposes a certain correspondence, related to the map $\Lambda$ in \eqref{pre defined map for thm1}, which will be assumed for Theorem \ref{thm 1}.
\begin{wh} \label{llast hyp}
Let $g \in \tG/G$ be given. 
Denote by $\chi_g \in  (\cS_{\vp, \scn}/\cS_{\tvp, \scn})^\vee$ the image of $g$ under the map $\Lambda.$
We assume that  ${^g}\sigma$ corresponds to $\rho \chi_g$ via the bijection \eqref{internal st} for $G.$
\end{wh}
\begin{rem} \label{rem for 4th wh}
Working Hypothesis \ref{llast hyp} is discussed for quasi-split groups in a different setting, and deduced by means of the conjectural endoscopic character identity (see \cite[Theorem 2.6 and Lemma 3.13]{xu15}).
\end{rem}

From Section \ref{restriction on p-adic and finite}, 
we recall the set
$I(\rho) = \{\chi \in (\cS_{\vp, \scn}/\cS_{\tvp, \scn})^\vee : \rho \chi \s \rho \},$ and consider the quotient group $(\cS_{\vp, \scn}/\cS_{\tvp, \scn})^\vee/I(\rho).$ The first bijection is as follows.

\begin{thm} \label{thm 1}
Suppose that Working Hypotheses \ref{wh temp llc}, \ref{shahidi's conj}, and \ref{last hyp} are valid for $G$ and $\tG,$ and further that Working Hypothesis \ref{llast hyp} holds.
Then we have the following isomorphism of finite abelian groups
\[
\lambda: \tG/\tG_{\si} \overset{\s}{\longrightarrow} (\cS_{\vp, \scn}/\cS_{\tvp, \scn})^\vee \big/ I(\rho).
\]
\end{thm}
\begin{proof}
Consider the following homomorphism 
\[
\bar\Lambda: \tG/G \longrightarrow   (\cS_{\vp, \scn}/\cS_{\tvp, \scn})^\vee/I(\rho).
\]
Since $\bar\Lambda$ is surjective due to the definition of $\Lambda$ in \eqref{pre defined map for thm1}, it suffices to show that $\ker  \bar\Lambda = \tG_\sigma/G.$
To this end, for any $g \in \tG/G$ with $\Lambda(g)=\chi_g \in I(\rho),$ due to  Working Hypothesis \ref{llast hyp}, ${^g}\si$ corresponds to $\rho$ via \eqref{internal st}. Since $\si$ corresponds to $\rho$ via \eqref{internal st}, it follows that ${^g}\si \s \si,$ and thus $g \in \tG_{\si}/G.$ 
On the other hand, for any $g \in \tG_{\si}/G,$  ${^g}\si$ corresponds to $\rho\chi_g$ via \eqref{internal st}. Since $\si \s {^g}\si,$ we have $\rho\chi_g \s \rho,$ and thus $\chi_g \in I(\rho).$
\end{proof}

\begin{rem}
Since $\tG/ \tG_{\si}$ simply transitively acts on $\Pi_{\ts}(G)$ (Proposition \ref{simply transitive on the set}), Theorem \ref{thm 1} immediately implies that there is a bijection between $\Pi_{\ts}(G)$ and $(\cS_{\vp, \scn}/\cS_{\tvp, \scn})^\vee \big/ I(\rho),$ sending $\si'\in\Pi_{\ts}(G)$ with $\si'={^g}\si$ for some $g \in \tG_\si/G$ to $\chi_g \in (\cS_{\vp, \scn}/\cS_{\tvp, \scn})^\vee \big/ I(\rho).$
\end{rem}

Combining the map \eqref{surj} and the isomorphism $H^1(W_F, \widehat {(\tbG/\bG)}) \s \big((\tbG/\bG)(F)\big)^\vee,$ it is noted that $\chi \in I(\ts)$ (see the definition \eqref{X(sigma)}) gives a 1-cocycle $a_\chi \in H^1(W_F, \widehat{\tbG/\bG})$ such that $\tvp a_\chi \s \tvp.$
We then have the following embedding 
\begin{equation} \label{I and X}
I(\ts) \hookrightarrow \{ a \in H^1(W_F, \widehat {(\tbG/\bG)}) : a \tvp \s \tvp \mbox{ in } \widehat{\tbG} \}\big/\imag\big(Z(\widehat \bG)^\Gamma \rightarrow H^1(W_F, \widehat{\tbG/\bG})\big) = X(\tvp).
\end{equation}
Note that this injection can also follow from the desiderata of the local Langlands correspondence (cf., \cite[Section 10]{bo79}).
We denote by $\bar I(\ts)$ the image 
of $I(\ts)$ in $X(\tvp)$ via \eqref{I and X}, and consider the quotient group $X(\tvp)/ \bar I(\ts).$
Using the isomorphism $X(\tvp) \s  \cS_{\vp, \scn}/\cS_{\tvp, \scn}$ in \eqref{imp exact}, we consider the following surjective homomorphism
\begin{equation} \label{for thm 2}
\widehat{\Lambda}: \cS_{\vp, \scn}/\cS_{\vp, \scn} \longrightarrow X(\tvp)/ \bar I(\ts).
\end{equation}

The {last} working hypothesis propose a counterpart of Working Hypothesis \ref{llast hyp}, related to the map $\widehat{\Lambda}$ in \eqref{for thm 2}, which will be assumed for Theorem \ref{thm 2}.
\begin{wh} \label{really llast hyp}
Let $s \in \cS_{\vp, \scn}/\cS_{\tvp, \scn}$ be given.
Denote by $\chi_s \in (\tG/G)^\vee$ the image of $s$ under the composition of the isomorphism $X(\tvp) \s  \cS_{\vp, \scn}/\cS_{\tvp, \scn}$ in \eqref{imp exact} and the maps in \eqref{for pre hom}.
We assume that $\tsigma \chi_s$ corresponds to ${^s}\trho$ via the bijection \eqref{internal st} for $\tG.$
\end{wh}
\begin{rem}
Working Hypothesis \ref{really llast hyp} is discussed for quasi-split groups in a different setting (see \cite[Sections 6.2 and 6.3]{xu15}).
In addition, this hypothesis implies that
the representation $\tsigma \chi_s$ lies in $\Pi_{\tvp}$ and $\Pi_{\tsigma}(G)=\Pi_{\tsigma \chi_s}(G) \subset \Pi_{\vp}(G).$ 
In fact, one can notice that the character $\chi_s$ is unitary, since it is given by an element in the finite abelian group $\{ a \in H^1(W_F, \widehat {(\tbG/\bG)}) : a \tvp \s \tvp \mbox{ in } \widehat{\tbG} \}$ corresponding to $s$ through  \eqref{for pre hom} where the local Langlands correspondence for $\widehat {(\tbG/\bG)}$ is applied.   
\end{rem}
As above, recalling the set $(\cS_{\vp, \scn})_{\trho} = \{s \in \cS_{\vp, \scn} : {^s}\trho \s \trho \},$
we consider the quotient $\cS_{\vp, \scn}/ (\cS_{\vp, \scn})_{\trho}.$ The other bijection is as follows.
\begin{thm} \label{thm 2}
Suppose that Working Hypotheses \ref{wh temp llc}, \ref{shahidi's conj}, and \ref{last hyp} are valid for $G$ and $\tG,$ and further that Working Hypothesis \ref{really llast hyp} holds.
Then, there is an isomorphism of finite abelian groups
\[
\widehat \lambda:  \cS_{\vp, \scn}/ (\cS_{\vp, \scn})_{\trho} \overset{\s}{\longrightarrow} X(\tvp)/\bar I(\ts).   
\]
\end{thm}
\begin{proof}
It suffices to show that the kernel of $\widehat \Lambda$ in \eqref{for thm 2} equals $(\cS_{\vp, \scn})_{\trho}/\cS_{\tvp, \scn}.$
Let $s \in \cS_{\vp, \scn}/\cS_{\tvp, \scn}$ be given such that $\widehat \lambda(s)=1.$ 
Then, recalling the image $\chi_s$ in $(\tG/G)^\vee$ of $s$ as described in Working Hypothesis \ref{really llast hyp}, we note that $\tsigma \chi_s \s \tsigma.$ Since $\trho$ corresponds to $\ts$ under the bijection \eqref{internal st}, Working Hypothesis \ref{really llast hyp} yields $\trho \s {^s}\trho,$ which implies that $s \in (\cS_{\vp, \scn})_{\trho}/\cS_{\tvp, \scn}.$ On the other hand,  since $\trho \s {^s}\trho$ for any element $s \in (\cS_{\vp, \scn})_{\trho}/\cS_{\tvp, \scn},$ we have $\tsigma \chi_s \s \tsigma$ and $\chi_s \in I(\ts)$ due to Working Hypotheses \ref{last hyp} and \ref{really llast hyp}. It follows from \eqref{I and X} that $\widehat \lambda(s)=1.$ This completes the proof. 
\end{proof}
\begin{rem} \label{rem for X-orbit}
We note from Theorem \ref{thm 1} that the subset $\Pi_{\ts}(G) \subset \Pi_{\vp}(G)$  is controlled by the quotient $(\cS_{\vp, \scn}/\cS_{\tvp, \scn})^\vee.$
Furthermore, the set $X(\tvp)$ acts on the $L$-packet $\Pi_{\tvp}(\tG)$ by the character twisting via the maps in \eqref{for pre hom}, i.e., for given $a \in X(\tvp)$ and $\ts \in \Pi_{\tvp}(\tG),$ the pair $(a, \ts)$ maps to $\ts \chi_a$ for $\chi_a \in (\tG/G)^\vee.$ 
Therefore, the quotient $X(\tvp) / \bar{I}(\ts)$ in Theorem \ref{thm 2} 
(hence, the set $\Pi_{\rho}(S_{\tvp, \scn})$ which is in bijection with $\cS_{\vp, \scn}/ (\cS_{\vp, \scn})_{\trho}$) is in bijection with the $X(\tvp)$-orbit, $\{\ts \chi_a : a \in X(\tvp) \},$ of $\ts$ in  $\Pi_{\tvp}(\tG).$
\end{rem}
\begin{rem} \label{remark for lambda and lambda hat}
It should be mentioned that, from their definitions, both $\lambda$ and $\widehat \lambda$ rely on the bijection \eqref{internal st}, which is conjectured to characterize the internal structures of $L$-packets $\Pi_{\tvp}(\tG)$ and $\Pi_{\vp}(\G).$ 
Thus, both maps depend on the choices of $\si \in \Pi_{\vp}(\G)$ corresponding to $\rho \in \Irr(\cS_{\vp, \scn}, \chi_{\bG})$ and $\ts \in \Pi_{\tvp}(\tG)$ corresponding to $\trho \in \Irr(\cS_{\tvp, \scn}, \chi_{\tbG}).$
Specially, if  $\bG$ and $\tbG$ are quasi-split and if $L$-parameters $\vp$ and $\tvp$ are tempered, the canonical choices are made by fixing a Whittaker datum so that $\rho$ and $\trho$ are taken to be the trivial character $\mathbbm{1}$ (see also Remark \ref{rem for int structure}).
\end{rem}
\subsection{An equality and multiplicity formula} \label{both multi}
Using two bijections constructed in Theorems \ref{thm 1} and \ref{thm 2}, 
we obtain an equality of multiplicities in two sides (Theorem \ref{mult thm}) and a general formulation of the multiplicity (Theorem \ref{multi generalized}).

Fix a lifting $\ts \in \Irr_{\temp}(\tG)$ of $\si \in \Irr_{\temp}(G)$ and $\trho \in \Irr(\cS_{\tvp, \scn})$ with $\trho \subset \Res_{\cS_{\tvp, \scn}}^{\cS_{\vp, \scn}} (\rho).$ Theorems \ref{thm 1} and \ref{thm 2} yield the following theorem.
\begin{thm} \label{mult thm}
Suppose that Working Hypotheses \ref{wh temp llc}, \ref{shahidi's conj}, and \ref{last hyp} are valid for $G$ and $\tG,$ and further that Working Hypotheses  \ref{llast hyp} and \ref{really llast hyp} hold.
With the above notation, we have
\[
\langle \sigma, \ts \rangle_{G} = \langle \trho, \rho \rangle_{\cS_{\tvp, \scn}}.
\]
\end{thm}
\begin{proof}
We first recall that
\[
|I(\ts)| = |\Pi_{\ts}(G)| \cdot \langle \sigma, \ts \rangle_{G}^2
\]
and $|X(\tvp)|=|\cS_{\vp, \scn}/\cS_{\tvp, \scn}|.$ 
Then, since $\bar I(\ts) \s I(\ts)$ as in \eqref{I and X},  
the bijection in Theorem \ref{thm 1} yields 
\begin{equation} \label{equals for the equal}
\frac{|\bar I(\ts)|}{\langle \sigma, \ts \rangle_{G}^2} = |\Pi_{\ts}(G)| = \frac{|\cS_{\vp, \scn}/\cS_{\tvp, \scn}|}{|I(\rho)|} = \frac{|X(\tvp)|}{|I(\rho)|}.
\end{equation}
Note that Theorem \ref{thm 2} provides
\begin{equation} \label{equals for the equal 2}
\frac{| X(\tvp) |}{| \bar I(\ts) | }=| \Pi_{\rho}(\cS_{\tvp, \scn})|=\frac{| I(\rho)|}{\langle \trho, \rho \rangle_{\cS_{\tvp, \scn}}^2}.
\end{equation}
Due to \eqref{equals for the equal} and \eqref{equals for the equal 2}, we thus have
\[
\frac{|\bar I(\ts)|}{\langle \sigma, \ts \rangle_{G}^2}=\frac{|\bar I(\ts)|}{\langle \trho, \rho \rangle_{\cS_{\tvp, \scn}}^2},
\]
which implies that $\langle \sigma, \ts \rangle_{G}^2 = \langle \trho, \rho \rangle_{\cS_{\tvp, \scn}}^2.$ 
This completes the proof.
\end{proof}
\begin{rem} \label{rem related to xu}
The equality in Theorem \ref{mult thm} has been earlier discussed for unitary principal series of quasi-split unitary groups in \cite[Section 4]{keys87} as well as for quasi-split groups with some other hypotheses in \cite[Section 6]{xu15}. 
\end{rem}

Next, we obtain the following formula for the multiplicity in the restriction in terms of some notions in Langlands dual side, using Clifford's theorem (see \cite[Proposition 20.8]{jamesliebeck01}).
\begin{thm} \label{multi generalized}
Assume as in Theorem \ref{mult thm}.
For any $\trho \in \Irr(\cS_{\tvp, \scn})$ with $\trho \subset \Res_{\cS_{\tvp, \scn}}^{\cS_{\vp, \scn}} (\rho),$
\[
\langle \sigma, \ts \rangle_{G} = \frac{\dim \rho}{\dim \trho} |\Pi_{\rho}(\cS_{\tvp, \scn})|^{-1}.
\]
\end{thm}
\begin{proof}
Due to Theorem \ref{mult thm}, it suffices to show that
\[
\langle \trho, \rho \rangle_{\cS_{\tvp, \scn}} = \frac{\dim \rho}{\dim \trho} |\Pi_{\rho}(\cS_{\tvp, \scn})|^{-1}.
\]
This is a consequence of Clifford's theorem which applies to the situation that $\cS_{\tvp, \scn}$ is a normal subgroup of a finite group $\cS_{\vp, \scn}$ whose factor group is abelian.
\end{proof}
\begin{rem} \label{rem for gen from hs}
We note from Clifford's theorem that $\dim \trho_1 = \dim \trho_2$ for any $\trho_1, \trho_2 \subset \Res_{\cS_{\tvp, \scn}}^{\cS_{\tvp, \scn}} (\rho).$ 
Recall that $\rho \in \Irr(\cS_{\vp, \scn}, \zeta_{\bG} )$ and $\trho \in  \Irr(\cS_{\tvp, \scn}, \zeta_{\tbG} )$ correspond to $\si$ and $\ts,$ respectively, due to the internal structure \eqref{internal st} for $G$ and $\tG.$ 
Since $\trho \subset \Res_{\cS_{\tvp, \scn}}^{\cS_{\tvp, \scn}} (\rho),$ it is clear that
\[
\langle \sigma, \ts \rangle_{G} = \frac{\dim \rho}{\dim \trho} |\Pi_{\rho}(\cS_{\tvp, \scn})|^{-1}.
\] 
\end{rem}

\begin{rem}
When $\cS_{\vp, \scn}$ is abelian, Theorem \ref{multi generalized} yields that $\langle \sigma, \ts \rangle_{G} $ is always equal to 1. 
Further, if $\tbG$ is quasi-split (so is $\bG$) and if $\cS_\vp$ is abelian, then it also follows that $\langle \sigma, \ts \rangle_{G} $ is always equal to 1.
\end{rem}
\subsection{Hiraga and Saito's work for $\SL_m(D)$ revisited} \label{section for SL}
In this section, 
we shall observe that Theorem \ref{multi generalized} generalizes the Hiraga and Saito's work in \cite{hs11} on the multiplicity in restriction for the case of inner forms of $\GL_n$ and $\SL_n$ toward the case of arbitrary connected reductive algebraic groups $\tbG$ and $\bG$ with the condition \eqref{cond on G}. 

Denote by $D$ a central division algebra of dimension $d^2$ over $F$ (possibly $D=F,$ in which case $d=1$). 
Let $\GL_{m}(D)$ denote the group of all invertible elements of $m \times m$ matrices over $D,$ and $\SL_m(D)$ the subgroup of elements in $\GL_m(D)$ with reduced norm equal to $1.$ 
By abuse of notation, we write $\GL_m(D)$ and $\SL_m(D)$ for their algebraic groups over $F$ as well.
Set $n=md.$ Then, any $F$-inner forms of $\GL_n$ and $\SL_n$ 
are of the form ${\GL}_m(D)$ and ${\SL}_m(D),$ respectively (see \cite[Sections 2.2 \& 2.3]{pr94}).
We have $\widehat{\GL_m(D)} = {\GL}_n(\CC)$ and $\widehat{\SL_m(D)} = {\PGL}_n(\CC),$ since $\Gamma$ acts trivially.
Due to \cite{abps13, gk82, ht01, he00, hs11, scholze13}, the local Langlands conjecture in \eqref{llc bij} and the conjectural internal structure of $L$-packets in \eqref{internal st} for all inner forms of $\SL_n$ and $\GL_n$ are known.
In particular, from \cite[Chapter 12]{hs11} and \cite[Section 3]{abps13},
given $\vp \in \Phi_{\temp}(\SL_m(D))$ and its lifting $\tvp \in \Phi_{\temp}(\GL_m(D))$ such that $\Pi_{\tvp}(\GL_m(D)) = \{ \ts  \},$  we have a bijection
\begin{equation} \label{bij sl}
\Pi_{\vp}\big({\SL}_m(D)\big) \overset{1-1}{\longleftrightarrow} \Irr\big(\cS_{\vp, \scn}(\widehat{{\SL}_m(D)}), ~ \zeta_{{\SL}_m(D)} \big),
\end{equation}
such that the isomorphism
\begin{equation*} 
V_{\ts} ~ ~ \s 
\bigoplus_{\rho \in \Irr\big(\cS_{\vp, \scn}(\widehat{\SL_m(D)}), ~\zeta_{\SL_m(D)} \big)} \rho \boxtimes  \si_{\rho}
\end{equation*}
as representations of $\cS_{\vp, \scn}(\widehat{\SL_m(D)}) \times \SL_m(D)$ holds,
where $\si_{\rho}$ denotes the image of $\rho$ via the bijection \eqref{bij sl} (see \cite[Lemma 12.6]{hs11}).
It then follows from \eqref{dim=multi} and \cite[p.5]{hs11} that 
\begin{equation} \label{dim=dim=multi}
\langle \sigma, \ts \rangle_{\SL_m(D)} = \dim \xi_{\si} = \dim \rho_{\si}.
\end{equation}
Further, for any $\si_1, \si_2 \in \Pi_{\vp}(\SL_m(D)),$ we have
$
{\dim}\rho_{\si_1} = {\dim}\rho_{\si_2}
$ 
(see also Remark \ref{dim1=dim2}).

Therefore, Theorem \ref{multi generalized} generalizes the Hiraga and Saito's work in \eqref{dim=dim=multi} to the case of arbitrary connected reductive algebraic groups $\tbG$ and $\bG$ with the condition \eqref{cond on G}. 
To be precise,
when $\tbG = \GL_m(D)$ and $\bG=\SL_m(D),$ we always have
\[
\dim \trho = 1 ~~\text{ and } ~~  |\Pi_{\rho}(\cS_{\tvp, \scn})|=1.
\]
It follows from Theorem \ref{multi generalized} and Remark \ref{rem for gen from hs} that
\[
\langle \sigma, \ts \rangle_{\SL_m(D)} =\dim \rho,
\]
which coincides with  \eqref{dim=dim=multi}.
\subsection*{Acknowledgements}
The author would like to express his great appreciation to Tasho Kaletha for his insightful comments and fruitful discussions on this work.
He is also grateful to Jeff Adler, Wee Teck Gan, Wen-Wei Li, Dipendra Prasad, and Mark Reeder for their valuable suggestions and helpful communications.
The author wishes to thank the referee for a careful reading and many valuable comments and suggestions that have led to improvements in the manuscript.
This work was partially done during his visit at Max-Planck-Institut f\"{u}r Mathematik, Bonn in June and July 2016. The author thanks the institute for their generous support and stimulating research environment.
\appendix
\section{Examples} \label{examples}
We present some examples related to the results established in Section \ref{section of main results}.
We continue with the notation in the previous sections.
\begin{exa} \label{eg-sp11}
This example is based on \cite[Sections 7.6 and 7.7]{ch15}.
Let $\tbG=\GSp_{1,1}$ be the non-quasi-split inner form of $\GSp_4,$ and let $\bG=\Sp_{1,1}$ be the non-quasi-split inner form of $\Sp_4.$ 
Let $\tvp = \tvp_0 \oplus (\tvp_0 \otimes \chi) \in \Phi(\GSp_{1,1})$ be given, 
where $\chi$ is a quadratic character, $\tvp_0 \in \Phi(\GL_2)$ is primitive (by definition, 
$\tvp_0$ is not of the form $\Ind_{W_E}^{W_F} \tau$ for any non-trivial finite extension $E/F$ and irreducible representation $\tau$), 
and $\tvp_0 \not\s \tvp_0 \otimes \chi.$ 

We have
\begin{equation*} \label{decomp Sp11}
{\Res}^{\GSp_{1,1}}_{\Sp_{1,1}}(\ts'_1) =  {\Res}^{\GSp_{1,1}}_{\Sp_{1,1}}(\ts'_2)= \{ \si' \},
\end{equation*}
and $\ts'_2 \s \ts'_1 \chi.$
Moreover, from the fact that $X(\tvp) \s \{ \mathbbm{1}, \chi\}$ (see \cite[Proposition 6.3(iii)(b)]{gtsp10}), it follows that 
\begin{equation} \label{sepcial multi for Sp11}
I(\ts'_1) = I(\ts'_2) = \{ \mathbbm{1} \}.
\end{equation}
Thus, the $L$-packet $\Pi_{\vp}({\Sp}_{1,1})$ of $\Sp_{1,1}$ attached to the $L$-parameter $\vp$ is $\{ \si' \}.$
We recall from \cite[Sections 7.6 and 7.7]{ch15} that 
\[
1 \longrightarrow \mu_2(\CC) \longrightarrow S_{\tvp, \scn}(\widehat {{\GSp}_{1,1}}) \s (\ZZ/2\ZZ)^2 \longrightarrow S_{\tvp}(\widehat {{\GSp}_{1,1}}) \s \ZZ/2\ZZ \longrightarrow 1,
\]
\[
1 \longrightarrow \mu_2(\CC) \longrightarrow S_{\vp, \scn}(\widehat {{\Sp}_{1,1}}) \s \mathcal{D}_8  \longrightarrow S_{\vp}(\widehat {{\Sp}_{1,1}}) \s (\ZZ/2\ZZ)^2  \longrightarrow 1,
\]
where $\mathcal{D}_8$ denotes the dihedral group of order 8.
Further, we note that $\Irr(\mathcal{D}_8)$ consists of four 1-dimensional characters and one 2-dimensional irreducible representation. 
We denote by $\rho'$ the  2-dimensional irreducible representation.
Setting $\Irr(\mu_2(\CC))=\{ \mathbbm{1}, \sgn\},$ the map $\sigma' \mapsto \rho'$ from $\Pi_{\vp}({\Sp}_{1,1})$ to  $\Irr(S_{\vp, \scn}(\widehat{{\Sp}_{1,1}}), \sgn)$ provides an equality 
\begin{equation} \label{mutli 2}
\dim \rho' = 2,
\end{equation}
while the multiplicity $\langle \sigma', \ts'_i \rangle_{\Sp_{1,1}}$ in ${\Res}_{\Sp_{1,1}}^{\GSp_{1,1}}(\ts'_i)$ for $i=1,2$  satisfies
\begin{equation} \label{multi for Sp11 special}
\langle \sigma', \ts'_i \rangle_{\Sp_{1,1}} = 1.
\end{equation}

We now consider $(\GSp_{1,1}(F))_{\sigma'}$ which turns out to be equal to $\GSp_{1,1}(F)$ and 
\[
\tG / \tG_{\sigma'} = \{1\}.
\] 
We also note that
\[
I(\rho') = \{ \eta \in \big(\mathcal{D}_8\big/(\ZZ/2\ZZ)^2\big)^\vee : \rho' \eta \s \rho'\} = \big(\mathcal{D}_8\big/(\ZZ/2\ZZ)^2\big)^\vee,
\]
which implies that 
\[
(S_{\vp, \scn}/S_{\tvp, \scn})^\vee \big/ I(\rho').
\]
Hence, this coincides with Theorem \ref{thm 1}.

For Theorem \ref{thm 2}, we first have
\[
X(\tvp) = \{1, \chi    \} ~~ \text{ and } ~~
I(\ts') = {1}.
\]
We note that, for any Klein-four subgroup $H$ of $\mathcal{D}_8$ (in fact, there are two), we have
\[
{\Res}^{\mathcal{D}_8}_{H}(\rho') = \{\chi_1, \chi_2\},
\]
where $\chi_1$ and $\chi_2$ are distinct 1-dimensional characters of $H,$ since the trace of $\rho'$ vanishes outside such subgroup $H$ (see \cite[(20.13)(2)]{jamesliebeck01}).
Fix $\chi_1$ corresponding to $\ts'_1$ via the bijection between $\Pi_{\tvp}({\GSp}_{1,1})$ and  $\Irr(\cS_{\tvp, \scn}(\widehat{{\GSp}_{1,1}}), \sgn).$
Computing the stabilizer $(\mathcal{D}_8)_{\chi_1}=\{s \in \mathcal{D}_8 : {^s}\chi_1 = \chi_1   \} ,$ we have $\mathcal{D}_8/(\mathcal{D}_8)_{\chi_1} \s \ZZ/2\ZZ.$ 
Hence, this coincides with Theorem \ref{thm 2}.

Lastly, for the multiplicity in the restriction, given a lifting $\ts' \in \Pi_{\tvp}(\GSp_{1,1})$ of $\sigma' \in \Pi_{\vp}(\Sp_{1,1}),$  we have
\[
\langle \sigma', \ts' \rangle_{G}  = \frac{\dim \rho'}{\dim \chi_1} |\Pi_{\rho'}(S_{\tvp, \scn})|^{-1} = \frac{2}{1} 2^{-1} = 1,
\]
which coincides with \eqref{multi for Sp11 special}.
Hence, this coincides with Proposition \ref{multi generalized}. This applies to all the others in $\Pi_{\vp}(\Sp_{1,1}).$
\end{exa}

\begin{exa} \label{eg-sp4}
Let $\tbG=\GSp_4,$ 
$\bG=\Sp_{4},$ and $\tvp$ be given as above in Example \ref{eg-sp11}.  From \cite[Sections 7.6 and 7.7]{ch15}, we have
\[
S_{\tvp}(\widehat{{\GSp}_{4}}) \s \ZZ/2\ZZ, \quad S_{\vp}(\widehat{{\Sp}_{4}}) \s \ZZ/2\ZZ \times \ZZ/2\ZZ,
\]
\[
{\Res}^{{\GSp}_4}_{{\Sp}_4}(\ts_1) = \{ \si_1^+, \si_1^- \}, \quad {\Res}^{{\GSp}_4}_{{\Sp}_4}(\ts_2) = \{ \si_2^+, \si_2^- \},
\]
\[
\Pi_{\tvp}({\GSp}_4) = \{ \ts_1, \ts_2   \}, \quad \Pi_{\vp}({\Sp}_4) = \{\si_1^+, \si_1^-, \si_2^+, \si_2^- \}.
\]
From the bijection $\Pi_{\vp} \overset{1-1}{\longleftrightarrow} \Irr(S_{\vp}(\widehat{{\Sp}_4})),$ we correspond $\si=\si_1^+$ to $\rho=\mathbbm{1}.$ 
Recall that there is a bijection ${\GSp}_4(F)/({\GSp}_4(F))_{\si} \overset{1-1}{\longleftrightarrow} \Pi_{\ts_1}(\Sp_4).$
Then we note that 
\[
\{ \eta \in  \big(S_{\vp}(\widehat{{\Sp}_{4}})\big/S_{\tvp}(\widehat{{\GSp}_{4}})\big)^\vee : \mathbbm{1} \eta = \mathbbm{1}\} = \{ \mathbbm{1} \} 
\]
and we have
\[
{\GSp}_4(F)\big/({\GSp}_4(F))_{\si} \s \ZZ/2\ZZ, 
\]
\[
 \big(S_{\vp}\s \ZZ/2\ZZ,(\widehat{{\Sp}_{4}})/S_{\tvp}(\widehat{{\Sp}_{4}})\big)^\vee \big/ \{ \eta \in  \big(S_{\vp}(\widehat{{\Sp}_{4}})/S_{\tvp}(\widehat{{\Sp}_{4}}) \big)^\vee : \mathbbm{1} \eta = \mathbbm{1}\} \s \ZZ/2\ZZ.
 \]

Hence,  this coincides with  Theorem \ref{thm 1}.
Moreover, from the bijection $\Pi_{\tvp} \overset{1-1}{\longleftrightarrow} \Irr(S_{\tvp}(\widehat{{\GSp}_4})),$ we correspond $\ts=\ts_1$ to $\trho=\mathbbm{1}.$ Then we have
\[
X(\tvp)/\bar I(\ts) = \{ \mathbbm{1} \},~~  S_{\tvp}(\widehat{{\GSp}_4})\big/S_{\tvp}(\widehat{{\GSp}_4})_{\mathbbm{1}} =\{ 1 \}.
\]
Hence,  this coincides with Theorem \ref{thm 2}.

Lastly, for the multiplicity in the restriction, given a lifting $\ts$ of $\sigma,$   we have
\[
\langle \sigma, \ts \rangle_{G}  = \frac{\dim \mathbbm{1}}{\dim \mathbbm{1}} |\Pi_{\mathbbm{1}}(S_{\tvp, \scn})|^{-1} = \frac{1}{1} 1^{-1} = 1.
\] 
Hence,  this coincides with Proposition \ref{multi generalized}. This applies to all the others in $\Pi_{\vp}(\Sp_4).$
\end{exa}

\begin{exa} \label{eg-sl2}
Let $\tbG=\GL_2,$ $\bG=\SL_2,$ and $\tvp \in \Phi(\tG)$  be dihedral with respect to three quadratic extensions. Then we have 
\[
S_{\tvp}(\widehat{{\GL}_2}) = \{ 1 \}, \quad S_{\vp}(\widehat{{\SL}_2}) \s \ZZ/2\ZZ \times \ZZ/2\ZZ,
\]
\[
\Pi_{\tvp}({\GL}_2) = \{ \ts\}, \quad {\Res}^{{\GL}_2}_{{\SL}_2}(\ts) = \Pi_{\vp}({\SL}_4) = \{ \si_1, \si_2, \si_3, \si_4 \}.
\]
From the bijection $\Pi_{\vp} \overset{1-1}{\longleftrightarrow} \Irr(S_{\vp}(\widehat{{\SL}_2})),$ we correspond $\si=\si_1$ to $\rho=\mathbbm{1}.$ 
Note that there is a bijection ${\GL}_2(F)/({\GL}_2(F))_{\si} \overset{1-1}{\longleftrightarrow} \Pi_{\ts}(\SL_2).$
We then have
\[
{\GL}_2(F)\big/({\GL}_2(F))_{\si} \s \ZZ/2\ZZ \times \ZZ/2\ZZ,
\]
\[
\big(S_{\vp}(\widehat{{\SL}_2})/S_{\tvp}(\widehat{{\SL}_2})\big)^\vee \big/ \{ \eta \in  \big(S_{\vp}(\widehat{{\SL}_2})/S_{\tvp}(\widehat{{\SL}_2})\big)^\vee : \mathbbm{1} \eta = \mathbbm{1}\} \s \ZZ/2\ZZ \times \ZZ/2\ZZ.
\]
Hence,  this coincides with  Theorem \ref{thm 1}.
Moreover, from the bijection $\Pi_{\tvp} \overset{1-1}{\longleftrightarrow} \Irr(S_{\tvp}(\widehat{{\GL}_2})),$ we correspond $\ts=\ts_1$ to $\trho=\mathbbm{1}.$ Then we have
\[
X(\tvp)/\bar I(\ts) = \{ \mathbbm{1} \},~~  S_{\tvp}(\widehat{{\GL}_2})\big/S_{\tvp}(\widehat{{\GL}_2})_{\mathbbm{1}} =\{ 1 \}.
\]
Hence,  this coincides with Theorem \ref{thm 2}.

Lastly, for the multiplicity in the restriction, given a lifting $\ts$ of $\sigma,$   we have
\[
\langle \sigma, \ts \rangle_{G} = \frac{\dim \mathbbm{1}}{\dim \mathbbm{1}} |\Pi_{\mathbbm{1}}(S_{\tvp, \scn})|^{-1} = \frac{1}{1} 1^{-1} = 1.
\] 
Hence,  this coincides with Proposition \ref{multi generalized}. This applies to all the others in $\Pi_{\vp}(\SL_2).$
\end{exa}
\begin{rem}
From Example \ref{eg-sl2}, we note that two sizes $|I(\rho)|=1$ and $|I(\ts)|=4$ does not necessarily equal each other. 
This implies that $|\Pi_{\ts}(G)|=4$ does not need to be identical with $|\Pi_{\rho}(\cS_{\tvp, \scn})|=1.$
\end{rem}

\begin{exa} \label{eg-d1}
Let $\tbG=\GL_1(D),$ $\bG=\SL_1(D),$ where $D$ is the quaternion division algebra over $F,$ and $\tvp \in \Phi(\tbG)$ be as in Example \ref{eg-sl2}. Then we have 
\[
1 \longrightarrow \mu_2(\CC) \longrightarrow S_{\tvp, \scn}(\widehat {{\GL}_1(D)}) \s \ZZ/2\ZZ   \longrightarrow S_{\tvp}(\widehat {{\GL}_1(D)}) ={1} \longrightarrow 1,
\]
\[
1 \longrightarrow \mu_2(\CC) \longrightarrow S_{\vp, \scn}(\widehat {{\SL}_1(D)}) \s Q_8  \longrightarrow S_{\vp}(\widehat {{\SL}_1(D)}) \s \ZZ/2\ZZ \times \ZZ/2\ZZ \longrightarrow 1,
\]
where $Q_8$ denotes the quaternion group of order 8.
Recall that
\[
\Pi_{\tvp}({\GL}_1(D)) = \{ \ts'\}, ~~{\Res}^{{\GL}_1(D)}_{{\SL}_1(D)}(\ts') = \Pi_{\vp}({\SL}_1(D)) = \{ \si' \},
\]
\[
\Irr(Q_8) = \{ \chi_1, \chi_2, \chi_3, \chi_4, \rho'  \},
\]
where $\chi_i$'s are distinct 1-dimensional representations, and $\rho'$ is the 2-dimensional representation of $Q_8.$ 
From the bijection $\Pi_{\vp}(\SL_1(D)) \overset{1-1}{\longleftrightarrow} \Irr(S_{\vp, \scn}(\widehat{{\SL}_1(D)}), \sgn) = \{ \rho' \},$ we correspond $\si'$ to $\rho'.$ 
Note that there is a bijection ${\GL}_1(D)/({\GL}_1(D))_{\si'} \overset{1-1}{\longleftrightarrow} \Pi_{\ts'}(\SL_1(D)).$
We then have
\[
{\GL}_1(D)\big/({\GL}_1(D))_{\si'} = \{1\},~~ \big(Q_8/(\ZZ/2\ZZ)\big)^\vee \big/ \{ \eta \in  \big(Q_8/(\ZZ/2\ZZ)\big)^\vee : \rho'\eta = \rho'\} \s \{ \mathbbm{1} \},
\]
since  $\{ \eta \in  (Q_8/(\ZZ/2\ZZ))^\vee : \rho'\eta = \rho'\} = \{ \chi_1, \chi_2, \chi_3, \chi_4 \} .$
Hence,  this coincides with  Theorem \ref{thm 1}.
Moreover, from the bijection $\Pi_{\tvp} \overset{1-1}{\longleftrightarrow} \Irr(S_{\tvp, \scn}(\widehat{{\GL}_1(D)}), \sgn),$ we correspond $\ts'$ to $\sgn.$ Then we have
\[
X(\tvp)/I(\ts') = \{ \mathbbm{1} \},~~ Q_8/(Q_8)_{\rho'} =\{ 1 \}.
\]
Hence,  this coincides with  Theorem \ref{thm 2}.

Lastly, for the multiplicity in the restriction, given a lifting $\ts'$ of $\sigma',$   we have
\[
\langle \sigma', \ts' \rangle_{G} = \frac{\dim \rho'}{\dim \mathbbm{1}} |\Pi_{\mathbbm{1}}(S_{\tvp, \scn})|^{-1} = \frac{2}{1} 1^{-1} = 2.
\] 
Hence,  this coincides with  Proposition \ref{multi generalized}. This applies to all the others in $\Pi_{\vp}(\SL_1(D)).$
\end{exa}

\end{document}